\documentclass[a4paper,reqno]{amsart}
\usepackage{amscd, amssymb, pb-diagram, amsmath}
\usepackage{amsthm}
\usepackage{slashed}

\usepackage{tikz}
\usepackage[margin=3.5cm]{geometry}
\usepackage{color}
\usepackage{tensor}
\usetikzlibrary{arrows,matrix,backgrounds,decorations.markings}

\numberwithin{equation}{section}

\usepackage[all]{xy}

\DeclareFontFamily{OT1}{pzc}{}
\DeclareFontShape{OT1}{pzc}{m}{it}{<-> s * [1.2] pzcmi7t}{}
\DeclareMathAlphabet{\mathpzc}{OT1}{pzc}{m}{it}

\def\Aut{\operatorname{Aut}}
\def\End{\operatorname{End}}
\def\Dom{\operatorname{Dom}}

\def\supp{\operatorname{supp}}

\def\dim{\operatorname{dim}}

\def\id{\operatorname{id}}

\def\sup{\operatorname{sup}}
\def\max{\operatorname{max}}
\def\min{\operatorname{min}}

\def\id{\operatorname{id}}
\def\supp{\operatorname{supp}}

\def\d{\operatorname{d}}

\def\C{\mathbb{C}}

\def\R{\mathbb{R}}
\def\N{\mathbb{N}}
\def\Z{\mathbb{Z}}
\def\T{\mathbb{T}}

\def\JJ{\mathcal{J}}

\def\NN{\mathcal{N}}

\def\TT{\mathcal{T}}

\newcommand{\ep}{\varepsilon}

\renewcommand{\mid}{\,:\,}

\newtheorem{thm*}{Theorem}
\newtheorem{thm}{Theorem}[section]
\newtheorem{cor}[thm]{Corollary}

\newtheorem{lemma}[thm]{Lemma}
\newtheorem{prop}[thm]{Proposition}

\theoremstyle{definition}
\newtheorem{definition}[thm]{Definition}

\newtheorem{notation}[thm]{Notation}

\theoremstyle{remark}
\newtheorem{remark}[thm]{Remark}

\newtheorem{example}[thm]{Example}

\newcommand{\G}{\mathcal{R}}
\newcommand{\op}{\textnormal{op}}

\newcommand{\Dsla}{\slashed{D}}

\begin{document}

\date{\today}
\title{Wieler solenoids, Cuntz-Pimsner algebras and $K$-theory}

\author[Deeley]{Robin J. Deeley}
\address{Robin J. Deeley,   Department of Mathematics, University of Hawaii, 2565 McCarthy Mall, Keller 401A, Honolulu HI 96822}
\email{robin.deeley@gmail.com}
\author[Goffeng]{Magnus Goffeng}
\address{Magnus Goffeng, Department of Mathematical Sciences, Chalmers University of Technology and the University of Gothenburg, SE-412 96 Gothenburg, Sweden}
\email{goffeng@chalmers.se}
\author[Mesland]{Bram Mesland}
\address{Bram Mesland, Institut f\"{u}r Analysis, Leibniz Universit\"{a}t, Welfengarten 1, 30167 Hannover, Germany}
\email{brammesland@gmail.com }
\author[Whittaker]{Michael F. Whittaker}
\address{Michael F. Whittaker, School of Mathematics and Statistics, University of Glasgow, United Kingdom}
\email{Mike.Whittaker@glasgow.ac.uk}

\begin{abstract} 
We study irreducible Smale spaces with totally disconnected stable sets and their associated $K$-theoretic invariants. Such Smale spaces arise as Wieler solenoids, and we restrict to those arising from open surjections. The paper follows three converging tracks: one dynamical, one operator algebraic and one $K$-theoretic. Using Wieler's Theorem, we characterize the unstable set of a finite set of periodic points as a locally trivial fibre bundle with discrete fibres over a compact space. This characterization gives us the tools to analyze an explicit groupoid Morita equivalence between the groupoids of Deaconu-Renault and Putnam-Spielberg, extending results of Thomsen. The Deaconu-Renault groupoid and the explicit Morita equivalence leads to a Cuntz-Pimsner model for the stable Ruelle algebra. The $K$-theoretic invariants of Cuntz-Pimsner algebras are then studied using the Pimsner extension, for which we construct an unbounded representative. To elucidate the power of these constructions we characterize the KMS weights on the stable Ruelle algebra of a Wieler solenoid. We conclude with several examples of Wieler solenoids, their associated algebras and spectral triples.
\end{abstract}

\maketitle

\section*{Introduction}

Inspired by Williams \cite{williamsexp}, Wieler \cite{Wielerpaper} recently proved that irreducible Smale spaces with totally disconnected stable sets always arise as solenoids -- inverse limits associated with continuous surjections. Due to Wieler's characterization of such spaces we call them Wieler solenoids. In view of \cite{deaacaneudoaod}, Wieler solenoids arising from an open surjection can be studied by means of Cuntz-Pimsner algebras; a class of $C^*$-algebras that provide a different picture than the usual Smale space groupoid $C^*$-algebras. The purpose of this paper is to study the connection between Wieler solenoids defined from an open mapping and Cuntz-Pimsner algebras. This connection allows us to improve our understanding of the $K$-theoretic invariants of such Smale spaces. We do this by linking Putnam's stable Ruelle algebra \cite{Put_Alg, PS} with unbounded Kasparov theory on Cuntz-Pimsner algebras studied by Adam Rennie and the second and third listed authors \cite{goffmes, GMR}. 

Wieler's classification result states that any irreducible Smale space with totally disconnected stable sets can be written in the form of an inverse limit $X=V\xleftarrow{g}V\xleftarrow{g}\cdots$ where $g:V\to V$ is a continuous finite-to-one surjection satisfying two additional axioms, which are weakened versions of $g$ being locally expanding and open. For our construction, $g$ must be an open map, rather than the weakened version appearing in Wieler's second axiom. When $g$ is open and satisfies Wieler's first axiom, we show $g$ is a local homeomorphism. We discuss several different assumptions on $g$ in Section \ref{SmaleDisStaSet} and show that they are all equivalent to the condition that $g$ is an open map. 

Our main results revolve around a systematic study of the $C^*$-algebras associated with Wieler solenoids arising from an open mapping. We arrive at our results through combining dynamics, operator algebras, and $K$-theoretic invariants. In particular, we describe: i) the global unstable set $X^{u}(P)$ as a locally trivial bundle over the base space $V$, ii) the stable Ruelle algebra as a Cuntz-Pimsner algebra and iii) the $KK$-theory of the stable Ruelle algebra. An immediate consequence of the Cuntz-Pimsner model for the stable Ruelle algebra is the complete classification of KMS weights on the stable Ruelle algebra for the natural gauge dynamics of the Cuntz-Pimsner algebra.

The structure of a Smale space naturally occurs in hyperbolic and symbolic dynamics, see \cite{Put_Alg,Putnam notes, PutLift,put,PS,Rue1,Rue2,Tho, Wielerpaper, yisolenoid}. The study of Smale spaces through invariants of their $C^*$-algebras has been central to the theory since its very beginning. In \cite{Rue2}, Ruelle introduced $C^*$-algebra theory into the study of Smale spaces and studied their noncommutative dynamics using equilibrium (KMS) states. Continuing Ruelle's program, Putnam and Spielberg \cite{Put_Alg,PS} showed that the stable and unstable Ruelle algebras are separable, simple, stable, nuclear, purely infinite, and satisfy the Universal Coefficient Theorem (UCT). It follows that $K$-theory is a complete invariant for the stable and unstable Ruelle algebras using the Kirchberg-Phillips classification theorem \cite{Phi}. The computation of $K$-theory for these algebras is therefore of fundamental importance. One of our main results provides a new method of computing $K$-theory for the stable Ruelle algebra of a Wieler solenoid through a six term exact sequence.

Invariants of the Smale spaces themselves led Putnam to define a homology theory for Smale spaces \cite{put}. Inspired by $KK$-theoretic correspondences, together with Brady Killough, the first and fourth listed author used Putnam's homology theory to define dynamical correspondences for Smale spaces  \cite{DKWcorr}. This paper is a first step towards understanding dynamical correspondences from the point of view of $KK$-theory \cite{Kas1, Kas2}. The reason for taking this first step for Wieler solenoids is the presence of an explicit Morita equivalence between the stable Ruelle algebra and a Cuntz-Pimsner algebra over $C(V)$. This allows us to describe the noncommutative geometry (unbounded $KK$-theory, \cite{BJ, connesbook} see also \cite{M, MR}) and noncommutative dynamics (KMS weights \cite{ln, lrrw}) of the stable Ruelle algebra in an explicit way. Our motivation for doing so is to make the Smale space origin clear and susceptible to generalization.

The dynamics of expanding continuous surjections $g:V\to V$ is well studied in the literature \cite{aahr,deaacaneudoaod,raeburnetal,GMR,lrrw,Nek_crelle,Tho,Wal,williamsexp,yisolenoid}. This fact guarantees us that the tools we need are available. For instance, \cite[Chapter 4.4]{Tho} treated Smale spaces constructed from expansive mappings with a dense set of periodic points, see more details in Remark \ref{remarkaboutthomsen} below. Although the starting point for the paper is \cite{Wielerpaper}, the inspiration for several results describing the dynamical structure of Smale spaces with totally disconnected stable sets comes from work of Thomsen, see \cite{Tho}. In particular, one of our aims is to make explicit the Morita equivalence obtained from Remark 1.14, Lemma 4.16 and Theorem 4.19 in \cite{Tho}. Since this paper is a first step towards understanding the $KK$-theory of Smale spaces, we are careful to keep the paper self-contained and all of our constructions in noncommutative geometry and dynamics are explicit.

We now conclude the introduction with our main results and the organization of the paper. In Section \ref{Sec:Smale spaces and dyn}, we recall the relevant definitions of Smale spaces. Section \ref{SmaleDisStaSet} introduces Wieler's characterization of irreducible Smale spaces with totally disconnected stable sets. In particular, Wieler shows that such Smale spaces $(X,\varphi)$ always arise from a continuous surjection $g:V\to V$ of a compact Hausdorff space $V$ in the sense that $X$ is the inverse limit $X:=V\xleftarrow{g}V\xleftarrow{g}\cdots$ and $\varphi:X\to X$ the shift map. The following is the main dynamical result of our paper, and appears as Theorem \ref{localfibb}. This result should be compared to the structural results of Hurder-Clark-Lukina on matchbox manifolds \cite{homhcl,shapehcl}.

\begin{thm*}
\label{firstintrothm}
Suppose $(X,\varphi)$ is an irreducible Wieler solenoid arising from an open continuous surjection. Let $P\subseteq X$ be a finite $\varphi$-invariant set of periodic points and 
\[
X^u(P):=\{x\in X: \text{ there exists } p\in P \text{ such that } \lim_{n\to \infty}\d_X(\varphi^{-n}(x),\varphi^{-n}(p))=0\}.
\]
Then the projection map $\pi_0:X^u(P)\to V$, $(x_0,x_1,x_2,\ldots)\mapsto x_0$ defines a covering map. In particular, $\pi_0:X^u(P)\to V$ is a locally trivial bundle with discrete fibres.
\end{thm*}

In section \ref{sectioncpmodel}, we discuss Cuntz-Pimsner algebras associated with a continuous surjective local homeomorphism, following \cite{deaacaneudoaod}. This Cuntz-Pimsner algebra is isomorphic to a Deaconu-Renault groupoid algebra. We conclude by relating our findings to Wieler solenoids. Pavlov-Troitsky's results on branched coverings \cite{pavlovstrots} show that, if $g:V\to V$ satisfies Wieler's axioms, it is necessary that $g$ is a local homeomorphism in order to have an associated Cuntz-Pimsner algebra.

In Sections \ref{sectionfour} and \ref{sectionfive} we define an unbounded $K$-cycle representing the Pimsner extension of $O_E$ and show that it applies to the stable Ruelle algebra through an explicit Morita equivalence. The following result combines Theorem \ref{moreth} and Proposition \ref{stablecpa} to relate the stable Ruelle algebra $C^*(G^s(P))\rtimes \Z$ with the Cuntz-Pimsner algebra $O_E$, here $G^s(P)$ denotes the stable groupoid of the Smale space.

\begin{thm*}
\label{moritathmintor}
Let $(X,\varphi)$ be an irreducible Wieler solenoid arising from an open mapping $g:V \to V$ and suppose $P\subseteq X$ is a finite $\varphi$-invariant set of periodic points. The crossed product groupoid $G^s(P)\rtimes \Z$ is groupoid Morita equivalent to the Deaconu-Renault groupoid $\G_g$, which satisfies $O_E\cong C^*(\G_g)$. 

Moreover, the stable Ruelle algebra $C^*(G^s(P))\rtimes \Z$ is itself isomorphic to a Cuntz-Pimsner algebra $O_{\tilde{E}}$, defined using the $C(V)$-Hilbert $C^*$-module $\mathpzc{H}_{C(V)}:=L^2(X^u(P))_{C(V)}$ and the bi-Hilbertian $\mathbb{K}_{C(V)}(\mathpzc{H}_{C(V)})$-bimodule $\tilde{E}:=\mathpzc{H}_{C(V)}\otimes E\otimes_{C(V)}\mathpzc{H}^*$ .
\end{thm*}

To describe the $K$-theoretic invariants of $C^*(G^s(P))\rtimes \Z$ we make use of the Pimsner extension for $O_E$. The Pimsner extension describes $C^*(G^s(P))\rtimes \Z$ in the triangulated category $KK^{\T}$ using $C(V)$ (carrying the trivial action of the circle group $\T$) and $[E]\in KK_0^{\T}(C(V),C(V))$ defined by equipping $E$ with the $\T$-action $z\cdot \xi:=z\xi$, for $z\in \T\subseteq \C$ and $\xi\in E$. This allows us to compute $K$-theoretic invariants, $\T$-equivariant or not, from knowledge of $E$ and $C(V)$ alone. The form of the Pimsner extension is determined by a distinguished class $\partial\in KK_1^{\T}(C^*(G^s(P))\rtimes \Z,C(V))$ that we now describe by combining Theorem \ref{pimsnerthm} and Theorem \ref{productwithzup}.

\begin{thm*}
\label{thmthree}
Let $(X,\varphi)$ be an irreducible Wieler solenoid arising from an open mapping $g:V \to V$, $P \subseteq X$ a finite set of periodic points and let $C^*(G^s(P))\rtimes \Z$ be the stable Ruelle algebra. Then there is a six term exact sequence 
\[\begin{CD}
K_0(C(V))@>1-\theta_*>>K_0(C(V)) @>\iota_{\mathcal{R}*}>> K_0(C^*(G^s(P))\rtimes \Z)\\
@A\partial AA  @. @VV\partial V\\
K_1(C^*(G^s(P))\rtimes \Z) @<\iota_{\mathcal{R}*}<<K_1(C(V))@<1-\theta_*<< K_1(C(V))\\
\end{CD}. \]
Moreover, each connecting map is represented by an explicit unbounded Kasparov module.
\end{thm*}

We note that the element $\iota_\mathcal{R}\in KK^{\T}_{0}(C(V),C^*(G^s(P))\rtimes \Z)$ is induced from a stable inclusion (cf. Corollary \ref{stableisocor}), the $*$-homomorphism $\theta:C(V)\to C(V,M_N(\C))$ is defined from a partition of unity as in \eqref{omegadefin} and $\partial\in KK_1^{\T}(C^*(G^s(P))\rtimes \Z,C(V))$ is explicitly represented by the unbounded Kasparov module appearing in Theorem \ref{productwithzup}. 

In fact, a more general statement is true: The $C^*$-algebra $C^*(G^s(P))\rtimes \Z$ fits into an exact triangle with $C(V)$, as in Theorem \ref{productwithzup}. Using this description allows us to compute $K$-theoretic invariants in the sense of Corollary \ref{corollarpseqpage13}. Since the $K$-theory computation is $\T$-equivariant one can also use it to compute the $K$-theory of the stable algebra $C^*(G^s(P))$, see Remark \ref{uoneremark}. The fact that all $KK$-classes appearing in Theorem \ref{thmthree} are explicitly represented by unbounded Kasparov modules makes the result suitable for computations.

The stable Ruelle algebra $C^*(G^s(P))\rtimes\Z$ carries a dual $\T$-action, so we can study the KMS weights of the associated $\R$-action. In  Corollary \ref{kmsweightsonds}, we combine a result of Laca-Neschveyev \cite{ln} and the Morita equivalence of Theorem \ref{moritathmintor} to prove the following.

\begin{thm*}
Let $(X,\varphi)$ be an irreducible Wieler solenoid arising from an open mapping $g:V \to V$, $P \subseteq X$ a finite set of periodic points and $\pi_0$ the map from Theorem \ref{firstintrothm}.
For $\beta>0$, there is a bijective correspondence between the measures $\mu$ on $V$ satisfying $g^*\mu=e^\beta \mu$ and KMS$_\beta$ weights on $C^*(G^s(P)) \rtimes \Z$ via $\mu\mapsto \phi_\mu$, where the KMS$_\beta$ weight $\phi_\mu$ is defined on $C_c(\mathcal{G}^s(P)\rtimes \Z)$ by 
\[
\phi_\mu(f):=\int_{X^u(P)} f(x,0,x)\mathrm{d}(\pi_0^*\mu).
\]
The $C^*$-algebra $C^*(G^s(P))\rtimes \Z$ always admits at least one KMS$_\beta$ weight. If $(V,g)$ is mixing, there is a unique $\beta>0$ for which there exists a KMS$_\beta$ weight and that KMS$_\beta$ weight is unique.
\end{thm*}

Throughout the paper $n$-solenoids feature as our running example. In addition,
we consider two other examples in Section \ref{examplessection}. These include subshifts of finite type, where the results in the present paper, together with previous results of the second and third listed authors \cite{goffmes} prove that the $K$-homology of the stable and unstable Ruelle algebra is exhausted by explicit $\theta$-summable spectral triples whose phases are finitely summable Fredholm modules (see Proposition \ref{proponcktwosided}). We study self-similar groups in Subsection \ref{selfsimilarsubsec}. The construction of limit sets of regular self-similar groups provide a broad range of examples fitting into the framework of the paper.

In Section \ref{solenoidsubsec}, we study solenoids constructed from local diffeomorphisms $g:M\to M$ acting conformally on a closed Riemannian manifold $M$. In this case, we can construct explicit spectral triples on the stable Ruelle algebras that are not Kasparov products with the class $\partial$ in Theorem \ref{thmthree}. These spectral triples are defined from an elliptic log-polyhomogeneous pseudo-differential operator (in some cases acting on the GNS-space of the KMS-weight). 

\section*{Acknowledgments}

The authors thank Tyrone Crisp, Ian Putnam, Adam Rennie and Aidan Sims for interesting and insightful discussions. We are indebted to the universities of Copenhagen, Glasgow, Gothenburg, M${\rm \ddot{u}}$nster (Focus Programme on $C^*$-algebras, 2015), Warwick, and Wollongong and for their hospitality and support in facilitating this collaboration. We gratefully acknowledge additional travel support by the Centre for Symmetry and Deformation (Copenhagen). MG wishes to thank the Knut and Alice Wallenberg foundation for their support. MG was partially supported by the Swedish Research Council Grant 2015-00137 and Marie Sklodowska Curie Actions, Cofund, Project INCA 600398, BM was partially supported by EPSRC grant EP/J006580/2, and MW was partially supported by ARC grant DP130100490.

\section{Smale spaces and their $C^*$-algebras}
\label{Sec:Smale spaces and dyn}

In this section we will recall some well known facts about Smale spaces and their associated $C^*$-algebras, see \cite{Put_Alg,Putnam notes,put,PS,Rue2,Tho}. For a more detailed presentation we refer the reader to \cite{put}. 

\subsection{Smale spaces}\label{Sec:Smale}

A dynamical system $(X,\varphi)$ consists of a compact metric space $X$ and a continuous map $\varphi:X\to X$. A Smale space is a dynamical system in which $\varphi$ is a homeomorphism and the space can be locally decomposed into the product of a coordinate whose points get closer together under the map $\varphi$ and a coordinate whose points get farther apart under the map $\varphi$. Ruelle axiomatized the notion of a Smale space with the following definition.

\begin{definition}[{\cite[p.19]{put}, \cite{Rue1}}]\label{defSmaSpa}
Suppose $X$ is a compact metric space and $\varphi:X \to X$ is a homeomorphism. Consider the data $(X,\d_X,\varphi,[\cdot,\cdot],\lambda,\ep_X)$ where $\ep_{X} > 0$ and $0<\lambda < 1$ are constants and
\[ [\cdot,\cdot]:\{(x,y) \in X \times X : \d_X(x,y) \leq \ep_{X}\}\to X,\quad (x,y) \mapsto [x, y]\]
is a continuous mapping, called \emph{the bracket mapping}. We say that $(X,\d_X,\varphi,[\cdot,\cdot],\lambda,\ep_X)$ is a \emph{Smale space} if the following axioms hold
\begin{itemize}
\item[B1] $\left[ x, x \right] = x$,
\item[B2] $\left[ x, [ y, z] \right] = [ x, z]$ if both sides are defined,
\item[B3] $\left[ [ x, y], z \right] = [ x,z ]$ if both sides are defined,
\item[B4] $\varphi[x, y] = [ \varphi(x), \varphi(y)]$ if both sides are defined,
\item[C1] For $x,y \in X$ such that $[x,y]=y$, we have $\d_X(\varphi(x),\varphi(y)) \leq \lambda \d_X(x,y)$, and
\item[C2] For $x,y \in X$ such that $[x,y]=x$, we have $\d_X(\varphi^{-1}(x),\varphi^{-1}(y)) \leq \lambda \d_X(x,y)$.
\end{itemize}
We denote a Smale space by $(X,\varphi)$, the other data is taken to be implicit.
\end{definition}

To see that Ruelle's axioms define a local product structure, for $x \in X$ and $ 0 < \ep \leq \ep_{X}$, define 
\begin{align*}
 X^{s}(x, \ep) & := \{ y \in X : \d_X(x,y) < \ep, [y,x] =x \} \,\text{ and } \\ 
X^{u}(x, \ep) & := \{ y \in X : \d_X(x,y) < \ep, [x,y] =x \}.
\end{align*}
It follows that for $y,y' \in X^s(x,\ep)$ we have $\d_X(\varphi(y),\varphi(y')) \leq \lambda \d_X(y,y')$. We call $X^s(x,\ep)$ a \emph{local stable set} of $x$. Similarly, for $z,z' \in X^u(x,\ep)$ we have $\d_X(\varphi^{-1}(z),\varphi^{-1}(z')) \leq \lambda \d_X(z,z')$, and we call $X^u(x,\ep)$ a \emph{local unstable set} of $x$. The local product structure on a Smale space arises in the following way. For $ 0 < \ep \leq \ep_{X}$ and $x\in X$, the bracket mapping defines a mapping 
\begin{equation}
\label{prodcstruc}
[\cdot,\cdot]:X^{s}(x, \ep)\times X^{u}(x, \ep)\to X,
\end{equation}
which is a homeomorphism onto an open neighbourhood of $x$ (see \cite[Proposition 2.1.8]{put}).

We also note that if $x,y \in X$ with $\d_X(x,y)<\ep_X/2$, then $[x,y]$ is the unique point $X^s(x,\ep_X) \cap X^u(y,\ep_X)$ and $[y,x]$ is the unique point $X^s(y,\ep_X) \cap X^u(x,\ep_X)$, see Figure \ref{Bracket intersection}. This fact implies that $(X,\d_X,\varphi,[\cdot,\cdot],\lambda,\ep_X)$ is uniquely determined by $(X,\varphi)$ (up to changing $\ep_X$ and $\lambda$).
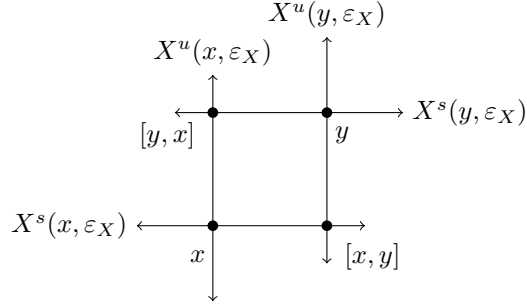
\begin{figure}[ht]
\begin{center}
\begin{tikzpicture}
\tikzstyle{axes}=[]
\begin{scope}[style=axes]
	\draw[<->] (-2,-1) node[left] {$X^s(x,\ep_X)$} -- (1,-1);
	\draw[<->] (-1,-2) -- (-1,1) node[above] {$X^u(x,\ep_X)$};
	\node at (-1.2,-1.4) {$x$};
	\node at (1.1,-1.4) {$[x,y]$};
	\pgfpathcircle{\pgfpoint{-1cm}{-1cm}} {2pt};
	\pgfpathcircle{\pgfpoint{0.5cm}{-1cm}} {2pt};
	\pgfusepath{fill}
\end{scope}
\begin{scope}[style=axes]
	\draw[<->] (-1.5,0.5) -- (1.5,0.5) node[right] {$X^s(y,\ep_X)$};
	\draw[<->] (0.5,-1.5) -- (0.5,1.5) node[above] {$X^u(y,\ep_X)$};
	\node at (0.7,0.2) {$y$};
	\node at (-1.6,0.2) {$[y,x]$};
	\pgfpathcircle{\pgfpoint{0.5cm}{0.5cm}} {2pt};
	\pgfpathcircle{\pgfpoint{-1cm}{0.5cm}} {2pt};
	\pgfusepath{fill}
\end{scope}
\end{tikzpicture}
\caption{The local coordinates of $x,y \in X$ and their bracket maps}
\label{Bracket intersection}
\end{center}
\end{figure}

Given a Smale space $(X,\varphi)$ and $x,y \in X$. We define equivalence relations by
\begin{align*}
x \sim_s y & \text{ whenever } \d_X(\varphi^n(x),\varphi^n(y)) \to 0 \text{ as } n \to \infty \text{ and} \\
x \sim_u y & \text{ whenever } \d_X(\varphi^{-n}(x),\varphi^{-n}(y)) \to 0  \text{ as } n \to \infty.
\end{align*}
We denote the stable equivalence class of $x \in X$ by  $X^s(x)$ and note that $X^s(x,\ep) \subset X^s(x)$. Similarly we denote the unstable equivalence class of $x \in X$ by  $X^u(x)$. A locally compact Hausdorff topology on $X^s(x)$ is generated by the open sets $$\{X^s(y,\ep) : y \in X^s(x), 0 <\ep< \ep_X\}.$$ A similar topology is defined in the unstable case. The reader should note that the topologies of $X^s(x)$ and $X^u(x)$ are in general different from the subspace topologies coming from the inclusions into $X$.

Lastly, we are interested in dynamical systems with topological recurrence conditions. In the following definition, we do not assume that $(X,\varphi)$ is a Smale space. 

\begin{definition}
\label{recprops} 
Suppose $(X,\varphi)$ is a dynamical system. We say $(X,\varphi)$ is
\begin{enumerate} 
\item \label{nonw} \emph{non-wandering} if, for all $x \in X$ and all open sets $U$ containing $x$ there exists an $N \in \N$ such that $\varphi^N(U) \cap U \neq \varnothing$;
\item \label{irred} \emph{irreducible} if, for all non-empty open sets $U,V \subseteq X$, there exists $N\in\N$ such that $\varphi^N(U) \cap V \neq \varnothing$;
\item \emph{mixing}\label{mixing} if, for all non-empty open sets $U,V \subseteq X$, there exists $N\in\N$ such that $\varphi^n(U) \cap V \neq \varnothing$ for all $n \geq N$. 
\end{enumerate} 
\end{definition}

In the previous definition it is clear that \eqref{mixing} $\implies$ \eqref{irred} $\implies$ \eqref{nonw}. In general, none of the reverse directions hold.
\subsection{$C^*$-algebras of Smale spaces}
\label{subsec:Algebras}

The first $C^*$-algebras associated to Smale spaces were defined by Ruelle in \cite{Rue2}; Ruelle considered the homoclinic algebra. The $C^*$-algebras of interest in this paper are associated with the stable and unstable equivalence relation. They are now called the stable and unstable algebras of a Smale space. Putnam \cite{Put_Alg} showed that there are natural crossed product $C^*$-algebras of the stable and unstable algebras, he called them the stable and unstable Ruelle algebras. These algebras generalize Cuntz-Krieger algebras in the sense that if the Smale space is a sub-shift of finite type, then the stable Ruelle algebra is Morita equivalent to the Cuntz-Krieger algebra $O_{\pmb{A}}$ where $\pmb{A}$ is the $0-1$ matrix defining the sub-shift (see more in Example \ref{firstckex} and Subsection \ref{cksubsec} below). All of these algebras are defined using groupoids. However, in \cite{Put_Alg} the groupoids were not \'{e}tale. Putnam and Spielberg \cite{PS} showed that up to Morita equivalence, one can take the groupoids to be \'{e}tale by restricting to an abstract transversal. We will use the construction from \cite{PS}.

Suppose $(X,\varphi)$ is a Smale space. It is well known that the set of $\varphi$-periodic points in $X$ is nonempty. This follows from the following two facts: the set of non-wandering points in $X$ is nonempty \cite[Appendix A.2]{Rue1} and is equal to the closure of the set of periodic points (see \cite[Section 7.3]{Rue1} and \cite[Lemma 3.8]{Bowen}). For a comprehensive treatment of these results, see \cite[Proposition 1.1.3, Theorem 4.4.1]{Putnam notes}. In particular, if $(X,\varphi)$ is non-wandering then the periodic points are dense in $X$. 

We choose a non-empty finite set of $\varphi$-invariant periodic points $P$. We define $X^{u}(P) = \cup_{p \in P} X^{u}(p)$ and endow this set with the locally compact and Hausdorff topology generated by the set $\{X^u(x,\ep) : x \in X^u(P),\; \ep\in (0,\ep_X]\}$. The stable groupoid is defined by
\begin{equation}
\label{stable groupoid def}
G^s(P) := \{(v,w) \in X \times X : v \sim_s w \textrm{ and } v,w \in X^u(P) \}
\end{equation}
and has unit space $X^u(P)$. For the groupoid $G^s(P)$, the partial product operation is given by $(v,w)(w,z):=(v,z)$ and inversion is given by $(v,w)^{-1}:=(w,v)$.
We construct a basis for an \'{e}tale topology on $G^s(P)$ as follows.

Suppose $v\sim_s w$ and choose $N \in \N$ such that $\varphi^{N}(v)\in X^{s}(\varphi^{N}(w),\ep_{X}/2)$. There is a relatively open neighbourhood $U$ of $w$ such that $\varphi^{N}(U)\subset  X^{u}(\varphi^{N}(w),\ep_{X}/2)$. This means that $[\varphi^{N}(x),\varphi^{N}(v)]\in X^{u}(\varphi^{N}(v),\ep_{X}/2)$ is defined for all $x\in U$. The map
\[
h^{s}_{N}:U\to X,\quad x\mapsto \varphi^{-N}[\varphi^{N}(x),\varphi^{N}(v)],
\]
is a local homeomorphism with $h^{s}_{N}(w)=v$. More precisely, if we take $\delta>0$ small enough and $U=X^u(w,\delta)$, then $h^{s}_{N}:X^u(w,\delta)\to X^u(v,\ep_X/2)$ is a homeomorphism onto its image. The mapping $h^s_N$ depends on the choice of $v$ and $w$ which we suppress from the notation as they will be clear from the context. Using this, a base of local neighbourhoods of $(v,w)\in G^{s}(P)$ is given by the sets
\begin{equation}
\label{dahood}
V^{s}(v,w,N,U):=\{(h^{s}_{N}(x),x): x\in U\subset X^{u}(w,\ep_{X}/2)\},
\end{equation}
where $N \in \N$ is such that $\varphi^{N}(v)\in X^{s}(\varphi^{N}(w),\ep_{X}/2)$ and the set $U$ is relatively open with $\varphi^{N}(U)\subset  X^{u}(\varphi^{N}(w),\ep_{X}/2)$. With this topology, $G^{s}(P)$ is an \'{e}tale groupoid. 

We define the stable algebra $C^*(G^s(P))$ as the reduced groupoid $C^*$-algebra associated with $G^{s}(P)$. In the case that $(X,\varphi)$ is irreducible, it is shown in \cite{PS} that $C^*(G^s(P))$ is strongly Morita equivalent to $C^*(G^s(Q))$ for any other finite $\varphi$-invariant set of periodic points $Q$. The same is true in the non-wandering case provided both $P$ and $Q$ intersect every irreducible component that arises in Smale's Decomposition Theorem \cite[Theorem 2.1.13]{put}. That is, for $P$ large enough and $X$ non-wandering, $C^*(G^s(P))$ is independent of $P$ up to stable isomorphism. 

The map $\alpha:=\varphi \times \varphi$ induces an automorphism of the $C^{*}$-algebra $C^*(G^s(P))$. The stable Ruelle algebra is the crossed product $C^*(G^s(P)) \rtimes_{\alpha} \Z$. The stable Ruelle algebra, as defined here, is strongly Morita equivalent to the stable Ruelle algebra originally defined by Putnam in \cite{Put_Alg}. In \cite{PS}, the Ruelle algebras were shown to be separable, stable, nuclear, purely infinite, and satisfy the UCT when $(X,\varphi)$ is irreducible and simple when $(X,\varphi)$ is mixing. The discussion above carries over to the unstable algebra using the Smale space $(X,\varphi^{-1})$ equipped with the opposite bracket $[x,y]_{\rm op}:=[y,x]$. The unstable groupoid with unit space $X^s(P)$ is denoted by $G^u(P)$.

\section{Wieler solenoids}
\label{SmaleDisStaSet}

Irreducible Smale spaces with totally disconnected stable sets were recently characterized by Wieler \cite{Wielerpaper}. We outline the main results \cite[Theorem A and B]{Wielerpaper} in Wieler's paper that every continuous surjection $g:V\to V$ satisfying Wieler's axioms defines a Smale space with totally disconnected stable sets, and that all such irreducible Smale spaces arise in this way. Assuming $g$ is open, we define a fibre bundle structure that will be used later in the paper. The bundle structure is similar to a construction in \cite{williamsexp} and related ideas appear implicitly in \cite{Tho}.

\begin{definition}
Suppose $f$ is a map on a compact metric space $Y$. We say that $f$ is \emph{locally expanding} if there exists constants $\delta>0$ and $\lambda>1$ such that $\d_Y(x,y)<\delta$ implies $\d_Y(f(x),f(y))\geq \lambda \d_Y(x,y)$. We say that $f$ is locally expansive for $\lambda>1$ within distance $\delta$.
\end{definition}

\begin{definition}[{\cite[p.2068]{Wielerpaper}}]
\label{wieleax}
Suppose $V$ is a compact metric space and $g:V\to V$ is a continuous surjection. We say \emph{$(V,g)$ satisfies Wieler's axioms} if there exists constants $\beta>0$, $K\in \N_+$, and $\gamma\in (0,1)$ such that the following hold: 
\begin{description}
\item[Axiom 1]\label{WAxiom1} If $v,w\in V$ satisfy $\d_V(v,w)<\beta$, then 
\[
\d_V(g^K(v),g^K(w))\leq\gamma^K \d_V(g^{2K}(v),g^{2K}(w)).
\]
\item[Axiom 2] For all $v\in V$ and $\ep \in (0,\beta]$
\[
g^K(B(g^K(v),\ep))\subseteq g^{2K}(B(v,\gamma\ep)).
\]
\end{description}
\end{definition}

\begin{remark}
\label{firstpavlovremark}
As discussed in \cite{Wielerpaper}, Wieler's axioms are weakenings of $g$ being locally expanding and open, respectively. Furthermore, note that \cite[Lemma 3.4]{Wielerpaper} shows that Wieler's Axiom 1 implies that $g$ is finite-to-one. In fact, if $g$ satisfies Wieler's axioms then combining \cite[Lemma 3.4]{Wielerpaper} with \cite[Theorem 2.9]{pavlovstrots} it follows that the map $g$ is a branched covering if and only if $g$ is open (see Subsection \ref{discussionong} below).
\end{remark}

Suppose $V$ is a compact metric space and $g:V\to V$ is a continuous surjection. We define
\begin{equation}
\label{solenodiconst}
X_V:=\{(v_i)_{i\in \N}\in V^\N: \; v_i=g(v_{i+1})\}
\end{equation}
along with a map $\varphi_g:X_{V}\to X_{V}$ given by
\begin{equation}
\label{solenodimap}
\varphi_g(v_0,v_1,\ldots):=(g(v_0),v_0,v_1,\ldots)=(g(v_0),g(v_1),g(v_2),\ldots).
\end{equation}

\begin{definition}
A \emph{Wieler solenoid} is a dynamical system of the form $(X_{V},\varphi_g)$ defined as above from a pair $(V,g)$ satisfying Wieler's axioms (see Definition \ref{wieleax}).
\end{definition}

We often use juxtaposition to write $(v_0,v_1,v_2\ldots)=v_0v_1v_2\cdots$. Note that $X_{V}$ is a compact metric space when equipped with the following metric (see \cite[page 2071]{Wielerpaper} for details):
\begin{equation}
\label{wielermetric}
\d_{X_{V}}(x,y)=\sum_{k=0}^K \gamma^{-k}\sup_{i\in \N}\gamma^i \d_V(x_{i+k},y_{i+k}),\quad\mbox{for}\quad x,y \in X_{V}.
\end{equation}

\begin{thm}[Wieler {\cite[Theorem A and B]{Wielerpaper}}]
\label{wielersthm} \quad
\begin{enumerate}
\item[(A)] Suppose $(V,g)$ satisfies Wieler's axioms. The associated Wieler solenoid $(X_{V},\varphi_g)$ is a Smale space with totally disconnected stable sets. If $(V,g)$ is irreducible, then $(X_{V},\varphi_g)$ is as well.
\item[(B)] Suppose $(X,\varphi)$ is an irreducible Smale space with totally disconnected stable sets. There exists an irreducible dynamical system $(V,g)$ satisfying Wieler's axioms such that $(X,\varphi)$ is conjugate to the Wieler solenoid $(X_{V},\varphi_g)$.
\end{enumerate}
\end{thm}

\begin{remark}
Suppose $(V,g)$ satisfies Wieler's axioms. It is unclear to the authors if applying the construction in the proof of \cite[Theorem \ref{wielersthm} (B)]{Wielerpaper} to the irreducible Smale space $(X_{V},\varphi_g)$ reproduces $(V,g)$ up to some suitable equivalence relation.
\end{remark}

Suppose $(X_V,\varphi_g)$ is a Wieler solenoid. The projection maps $\pi_k:X\to V$ are defined by $\pi_k(v_0,v_1,\ldots):=v_k$. We note that 
\begin{equation}
\label{pikrela}
\pi_k\circ \varphi_g=g\circ \pi_k,\quad \textnormal{and}\quad \pi_k\circ \varphi_g^{-1}=\pi_{k+1}.
\end{equation}

It natural to consider stronger conditions than Wieler's Axiom 1 and 2. For example, one could assume $g$ satisfies Axiom 1 and is open, one could assume $g$ is locally expanding and satisfies Axiom 2, or one could assume $g$ is locally expanding and open. The next two results show that each of these strengthenings of Wieler's axioms are equivalent to $g$ being a local homeomorphism with $g^K$ locally expanding.

\begin{lemma}
\label{locally expanding}
Suppose $V$ is a compact metric space and $g:V \rightarrow V$ is an open surjection satisfying Wieler's Axiom 1. Then $g$ is a local homeomorphism and $g^K:V \to V$ is locally expanding for $\gamma^{-K}$, where $\gamma\in (0,1)$ and $K \in \N_+$ appear in Wieler's axioms.
\end{lemma}

\begin{proof}
Let $\beta>0$, $K\in \N_+$, and $\gamma\in (0,1)$ be the fixed constants from Wieler's axioms. Suppose that $g^K$ is not locally expanding for $\gamma^{-K}$. For any $n \in \N$, there exists $v_n, w_n \in V$ such that 
\begin{enumerate}
\item $\d_V(v_n, w_n) < \frac{1}{n}$ and
\item $\d_V(v_n, w_n) > \gamma^K \d_V(g^K(v_n), g^K(w_n))$.
\end{enumerate}
By compactness, $\{v_n\}_{n\in \N}$ has a limit point, which we denote by $x$. Also let $x'$ denote a preimage of $x$ under $g^K$.

Take $0<r\le \frac{\beta}{2}$. The mapping $g$ is open, so $g^K(B_{r}(x'))$ is an open set. We can therefore choose $N \in \N$ such that $v_N$ and $w_N$ are in $g^K(B_{r}(x'))$. Since $v_N,w_N\in g^K(B_{r}(x'))$, there are pre-images $v'$ and $w'$ in $B_r(x')$ such that $g^K(v')=v_N$ and $g^K(w')=w_N$. By the triangle inequality, $\d_V(v', w')< 2r \le \beta$. Applying Wieler's first axiom to $v'$ and $w'$ gives
\[
\d_V(v_N, w_N) =\d_V(g^K(v'), g^K(w')) \le \gamma^K \d_V(g^{2K}(v'), g^{2K}(w'))=\gamma^K\d_V(g^K(v_N), g^K(w_N)).
\]
This is a contradiction to (2) above, so $g^K$ is locally expanding. If $g^K$ is locally expanding then $g^K$ is locally injective, so $g$ is locally injective. Since $g$ is open and $V$ is compact, it follows that $g$ is a local homeomorphism.
\end{proof}

\begin{lemma}
\label{open}
Suppose $V$ is a compact metric space and $g: V \rightarrow V$ satisfies Wieler's Axiom 2 and $g^K$ is locally expansive. Then $g$ is a local homeomorphism.
\end{lemma}

\begin{proof}
Since $g^K$ is locally expanding, it is locally injective. Thus, we need only show that it is open. Standard results imply that we need only show that for each $v \in V$ and $\ep>0$, there exists $\delta>0$ such that $B(g^K(v), \delta) \subseteq g^K(B(v, \ep))$. We can assume $\ep \le \beta$ and $g^K|_{B(g^K(v),\frac{\ep}{\gamma})}$ is injective. By Axiom 2, 
\[
g^K\left(B\left(g^K(v),\frac{\ep}{\gamma}\right)\right)\subseteq g^{2K}(B(v,\ep)).
\] 
Since $g^K|_{B(g^K(v),\frac{\ep}{\gamma})}$ is injective, we have 
\[
B\left(g^K(v),\frac{\ep}{\gamma}\right)\subseteq g^K(B(v,\ep)).
\]
It follows that $g^K$ is open and hence that $g$ is also open.
\end{proof}

\begin{remark}
\label{clobra}
By combining Lemma \ref{locally expanding} with \cite[Theorem 2.9]{pavlovstrots} we can conclude that if $g:V\to V$ is a branched covering satisfying Wieler's axioms then $g$ is open and, by Lemma \ref{locally expanding}, a finite sheeted covering (i.e. a surjective local homeomorphism), cf. Remark \ref{firstpavlovremark} and Subsection \ref{discussionong}.
\end{remark}

\begin{lemma} 
\label{bracketEqualsStableForSmallN}
Suppose $V$ is a compact metric space and $g:V \rightarrow V$ is a locally expanding continuous surjection satisfying Wieler's axioms. Let $K \in \N_+$ be the constant in Wieler's axioms and let $\delta >0$ be such that $(V,g^K)$ is locally expanding within distance $2\delta$. Suppose $M$ is a multiple of $K$ and $x=x_0x_1\cdots, y=y_0y_1\cdots \in X_{V}$ satisfy 
\begin{enumerate}
\item\label{bracket lemma 1} $\d_X(x,y) < \ep_X$ and
\item for each $m \le M$, $\d_V(x_m, y_m)< \delta$.
\end{enumerate}
Then, for each $m \le M$, $[ x,y ]_m =x_m$.
\end{lemma}

\begin{proof}
By assumption \eqref{bracket lemma 1}, $[x,y]$ is defined. We use induction on $m\le M$. For $m=0$, the result follows from the definition of the bracket, see \cite[Lemma 3.3]{Wielerpaper} and the discussion following its proof. Thus, we need to show that $[x,y]_m=x_m$ assuming that $[x,y]_{m-1}=x_{m-1}$. The construction of the bracket in \cite[proof of Lemma 3.3]{Wielerpaper} implies that $\d_V([x,y]_m,y_m)< \delta$. Moreover, by assumption, $\d_V(x_m, y_m)< \delta$ so the triangle inequality implies 
\begin{equation}\label{210_1}
\d_V([x,y]_m, x_m) < 2\delta.
\end{equation}
Using the induction hypothesis, we have 
\begin{equation}\label{210_2}
g^K([x,y]_m)=g^{K-1}([x,y]_{m-1}) = g^{K-1}(x_{m-1})=g^K(x_m).
\end{equation}
Thus, combining \eqref{210_1} and \eqref{210_2} with $g$ being locally expanding implies that 
\[
\lambda\d_V([x,y]_m,x_m) \leq  \d_V(g^K([x,y]_m),g^K(x_m))=0.
\]
We conclude that $[x,y]_m = x_m$.
\end{proof}

Suppose $V$ is a compact metric space and $g:V \rightarrow V$ is a map. Define
\[
V(N,g):=\{ (v_0, v_1, \ldots, v_N) \in V^{N+1} \mid g(v_i) = v_{i-1} \hbox{ for all }1 \le i \le N \}.
\]
Let $\pi_0^N$ denote the projection map $(v_0, v_1, \ldots, v_N) \mapsto v_0$.

\begin{lemma}
\label{morelocaltriv}
Suppose $g:V\to V$ is an onto mapping. Then, $g$ is a local homeomorphism if and only if for each $N\ge 0$, the projection map $\pi_{0}^N: V(N, g) \to V$ is a finite-to-one covering map.
\end{lemma}

\begin{proof}
Take $N>0$. The map $g$ is a local homeomorphism if and only if $g^N$ is. The space $V(N, g)$ is homeomorphic to the graph of $g^N$ in a way that is compatible with $\pi^N_0$. In particular, $\pi_{0}^N: V(N, g) \to V$ is a local homeomorphism if and only if $g^N$ is a local homeomorphism.
\end{proof}

\begin{thm}
\label{localfibb}
Let $(X,\varphi)$ be an irreducible Wieler solenoid defined from an open surjection $g:V\to V$. For any finite set of $\varphi$-invariant periodic points $P\subseteq X$, the mapping $\pi_0:X^u(P)\to V$ is a covering map.
\end{thm}

\begin{remark}
\label{remarkaboutthomsen}
A mapping $g:V\to V$ is called expanding in \cite{Tho} if it is open, surjective and there exists $\delta>0$ such that $\d_V(g^n(x),g^n(y))\leq \delta$ for all $n$ implies that $x=y$. In the proof of \cite[Theorem 4.19]{Tho}, Thomsen gives the proof of a statement similar to Theorem \ref{localfibb} under the assumption that $g$ is expanding and the set of periodic points in $X$ is dense.
\end{remark}

\begin{proof}[Proof of Theorem \ref{localfibb}]
We first show that $\pi_k:X^u(P)\to V$ is surjective for any $k$. Using \eqref{pikrela}, it suffices to prove that $\pi_0:X^u(P)\to V$ is surjective. Let $v\in V$, since $g$ is surjective, pick $x\in X$ such that $x_0=v$. Since $X^u(P)\subseteq X$ is dense, there exists $y\in X^u(P)$ such that $\d_X(x,y)<\ep_X$. Then $[x,y]\in X^u(P)$ and $[x,y]_0=x_0=v$ by Lemma \ref{bracketEqualsStableForSmallN}, so $[x,y]\in \pi_0^{-1}(v)$. 

Lemma \ref{locally expanding} implies the map $g$ is a local homeomorphism. By Lemma \ref{morelocaltriv}, it suffices to prove that there exists an $M \in \N$ such that $\pi_M:X^u(P) \to V$ is a covering map. We take $M$ and $\delta >0$ as in the hypotheses of Lemma \ref{bracketEqualsStableForSmallN}, such that $\delta$ is also sufficiently small so that if $x,y \in X$ satisfy $\d_V(x_m,y_m)<\delta$ for all $m \leq M$, then $\d_X(x,y)<\ep_{X}/4$. Note that we can take $M$ and $\delta$ such that $M$ is a multiple of the constant $K$ appearing in Wieler's axioms and that $(V,g)$ is locally expanding for $2\delta$.

Using the fact that $g$ is continuous and a short induction argument, we have that for fixed $v \in V$, there exists an open set $U \subseteq V$ such that $v \in U$ and 
\begin{equation}\label{M delta flattening}
\d_V(g^k(w_1), g^k(w_2)) < \delta \quad \text{ for any } w_1,w_2 \in U \text{ and each } 0\le k \le M.
\end{equation}
In particular, for such an open set $U$ we have $\pi_M^{-1}(U) \subseteq X^s(x, \ep_{X}) \times X^u(x, \ep_{X})$ for $x \in \pi_M^{-1}(v)$, where we identify $X^s(x, \ep_{X}) \times X^u(x, \ep_{X})$ with an open neighbourhood of $x$ in $X$ using the bracket mapping as in Equation \eqref{prodcstruc}. Suppose $v \in V$ and let $U \subseteq V$ be an open set satisfying \eqref{M delta flattening}. For any $x \in \pi_M^{-1}(v)$ define 
$$\Psi_{v}: \pi_M^{-1}(U) \rightarrow U \times (X^s(x, \ep_{X}) \cap X^u(P)),\quad
\Psi_{v}(z) =\big(\pi_M(z), [x, z]\big).$$ 

We first claim that $\Psi_v$ is well-defined. By the choice of $U$ and $M$, the bracket $[x, z]$ is defined, and is an element of $X^s(x, \ep_{X})$ since $\big[[x,z],x\big]=x$ by B1 and B3. Similarly, $[x, z] \in X^u(z,\ep_X) \subset X^u(P)$. Thus, $\Psi_v(z) \in U \times (X^s(x, \ep_X) \cap X^u(P))$. 

Next we show that $\Psi_v$ is independent of the choice of $x\in \pi_M^{-1}(v)$ in its definition. Suppose $x' \in \pi_M^{-1}(v)$. Then, by the choice of $M$, $\d_X(x, x') < \frac{\ep_X}{4}$. Since $x \sim_s x'$, we have $[x, z] \in X^s(x', \ep_X)$. Using B2 we have the identity
\[
[x, z] = [x', [x, z]] = [x',z].
\]
The mapping $\Psi_v$ is therefore independent of the choice of $x\in \pi_M^{-1}(v)$.

That $\Psi_{v}$ is continuous follows immediately since the projection and bracket maps are continuous. We now show that $\Psi_v$ is one-to-one. Suppose that $\Psi_v(z) = \Psi_v(z')$ for $z,z' \in \pi_M^{-1}(U)$. Then, since $\pi_M(z)=\pi_M(z')$ and $z \in X^s(z', \ep_X)$, using B2 and B1 we have
\[
z'= [ z,z']=[z, [ x,z']] = [z, [ x, z]] = [z, z]= z, 
\]
Thus $\Psi_v$ is one-to-one. The set 
\[
F_v:=\{ y \mid y=[x, z] \hbox{ for some }z \in \pi_M^{-1}(U) \}\subseteq X^s(x, \ep_X) \cap X^u(P),
\]
has the discrete topology because $X^s(x, \ep_X) \cap X^u(P)$ is discrete. We will show that the map $\Psi_v : \pi_M^{-1}(U) \rightarrow U \times F_v$ is a homeomorphism. We have shown that $\Psi_v$ is one-to-one, onto, and continuous. In fact, we have an explicit inverse. Let $\psi_v : U \times F_v \rightarrow \pi_M^{-1}(U)$ be defined by $\psi_v\big((w, y)\big) = [z', y]$ where $z'$ satisfies $\pi_M(z')=w \in U$ and $\d_V(z'_n, y_n) < \delta < \ep_X$ for each $n \le M$. The existence of $z'$ is guaranteed by \eqref{M delta flattening}.

We note that the bracket $[z', y]$ is well-defined by the properties of $U$. Furthermore, if both $z'$ and $z''$ satisfy $\pi_M(z')=w=\pi_M(z'')$, then $z'' \in X^s(z', \ep_X)$. Using B3 we compute
\[
[z', y]  =[ [z'', z'], y] = [z'', y].
\]
Hence $\psi_v$ is well-defined.

We show $\psi_v$ is continuous. Since $F_v$ has the discrete topology and the topology of local unstable sets coincides with the subspace topology, we need only show that given $\ep_{\psi}>0$, there exists $\delta_{\psi}>0$ such that if $(w, y)$ and $(\hat{w}, y)$ are in $U \times F_v$ with $\d_V(w, \hat{w})<\delta_{\psi}$, then 
\[
\d_X(\psi(w,y), \psi(\hat{w},y))=\d_X( [z', y ] , [\hat{z}', y ] ) < \ep_{\psi}
\]
where $z'$ satisfies the following: $\pi_M(z')=w \in U$ and $\d_V(z'_n, y_n) < \delta < \ep_X$ for each $n \le M$ and $\hat{z}'$ satisfies the analogous condition. 

Based on the definition of the metric on $X$, we need only show that there exists $\hat{\delta}_{\psi}>0$ such that, for each $n$, $\d_V( [z', y ]_n , [\hat{z}', y ]_n )< \ep_{\psi}$ whenever $\d_V(w, \hat{w})< \hat{\delta}_{\psi}$. We have that $[z', y ]$ and $[\hat{z}', y ]$ are both in the local unstable set of $y$. It follows that there exists $L\ge 0$ such that for any $n\ge L$, $\d_V(z_n', \hat{z}_n') < \ep_{\psi}$. For small $n$, since $g$ is a local homeomorphism, $\pi_M(z')=w$, and $\pi_M(\hat{z}')=\hat{w}$, there exists $\hat{\delta}_{\psi}>0$ such that 
\[
\d_V(z_n', \hat{z}_n') < \ep_{\psi}
\]
for $0\le n \le L$. This completes the proof that $\psi_v$ is continuous.

To see that $\psi_v$ is the inverse of $\Psi_v$, we first check that $\pi_M([ z' , y ])=w$. By construction, $\d_V(z'_n, y_n) < \delta$ for $0\le n\le M$ and $\d_X(z',y)< \ep_X$, so by Lemma \ref{bracketEqualsStableForSmallN} we have that
\[ 
[z' , y ]_{M} = z'_M = w,
\]
as desired. Furthermore, using B2 and B1 we compute
\[
(\psi_v \circ \Psi_v)(z)= [ z, [ x, z]] = [z, z] = z.
\]
Using B2 we deduce
\[
(\Psi_v \circ \psi_v)(w, y) = (\pi_M([ z' , y ]), [ x , [z', y]  ]) = (w, [ x , y ]) = (w, y). 
\]
Here we have used the fact that $y \in X^s(x, \ep_X)$ to conclude that $[x, y]=y$.
\end{proof}

\begin{example}
\label{solenioidexampleonaroundpage7}
There are several examples of open maps satisfying Wieler's axioms. In this example we will introduce a particularly simple class, the $n$-solenoids. Consider $V=S^1:=\R /\Z$ and $g(x)=nx \pmod 1$. We set the global Smale space constants to be $\ep_X=1/2$ and $\lambda=n$. For the remainder of this example we will abuse notation and write $nx$ for $nx \pmod 1$. The space
\[
X_{S^1}:=\{(x_0,x_1,x_2,\dots) : x_i \in [0,1), x_i-nx_{i+1}\in \Z \}
\]
is a compact metric space with respect to the product metric
\[
\d_{X_{S^1}}\big((x_0,x_1,x_2,\dots),(y_0,y_1,y_2,\dots)\big)=\sum_{i=0}^\infty n^{-i}\inf\{|x_i-y_i+k| : k\in\Z\}.
\]
We remark that this metric differs from that used in Wieler's construction, see Equation \eqref{wielermetric} on page \pageref{wielermetric}. The dynamics $\varphi_g$ from \eqref{solenodimap} is given by
\[
\varphi_g(x_0,x_1,x_2,\dots)=(nx_0,x_0,x_1,x_2,\dots)=(nx_0,nx_1,nx_2,\dots).
\]
Suppose $x,y$ are in $X_{S^1}$ with $\d_{X_{S^1}}(x,y)<\ep_X$. Let $t=x_0-y_0$, then the bracket map is defined by
\[
[x,y]:=(y_0+t,y_1+n^{-1}t,y_2+n^{-2}t,\dots).
\]
With these definitions in hand it is routine to verify the Smale space axioms. These details can also be found in Putnam's Smale space notes \cite[Section 3.4]{Putnam notes}.

More generally, any matrix $A\in M_d(\Z)$ induces a mapping $g_A:(S^1)^d\to (S^1)^d$. The mapping $g_A$ is a local homeomorphism exactly when $\det(A)\neq 0$ and $g_A$ is a homeomorphism when $|\det(A)|=1$. Moreover, if $\det(A)\neq 0$, $g_A$ satisfies Wieler's axioms if and only if $\|A^{-1}\|_{M_n(\R)}<1$. By \cite[Section 2.2]{lrrw}, the action of $g_A$ on $(S^1)^d$ is an example of the shift mapping acting on the limit space of a self-similar group. This construction is discussed further in Subsection \ref{selfsimilarsubsec}.
\end{example}

There are also examples that satisfy Wieler's axioms where the relevant map is not open, so our constructions do not apply to these examples. Examples 1 and 3 in \cite{Wielerpaper} are two such examples. An additional example is the Smale space associated with an aperiodic substitution tiling, the details are in \cite{AP}. A Cuntz-Pimsner model for the stable Ruelle algebra of an aperiodic substitution tiling is constructed in forthcoming work by Peter Williamson. In the next section we construct a Cuntz-Pimsner model over $C(V)$ for the stable Ruelle algebra of a Wieler solenoid defined from an open surjection $g:V\to V$. It is an interesting challenge to find ``good" Cuntz-Pimsner models for Ruelle algebras of more general Wieler solenoids, or even more general Smale spaces.

\section{Cuntz-Pimsner algebras and topological dynamics}
\label{sectioncpmodel}

In this section, we recall a construction from \cite{deaacaneudoaod} of Cuntz-Pimsner models describing the dynamics of a surjective local homeomorphism $g:V\to V$. We discuss the limitations of this assumption on $g$ in Subsection \ref{discussionong}, but emphasize that the results in this section do {\bf not} need $g$ to satisfy Wieler's axioms. For the general construction of Cuntz-Pimsner algebras, see \cite{Pimsnerspaper}.

\subsection{The Cuntz-Pimsner algebra of a local homeomorphism}
\label{TheCuntzPimsneralgebraofalocalhomeomorphism}

Consider a compact space $V$ and a surjective local homeomorphism $g:V\to V$. A map $g^{*}:C(V)\to C(V)$ is defined by $g^{*}(a):= a\circ g$. We consider $E:=C(V)$ as a right Hilbert module over itself via the action $g^{*}$ and the inner product
\begin{equation}
\label{innerproductl}
\langle \xi,\eta\rangle_E:=\mathfrak{L}(\bar{\xi}\eta),\quad\mbox{where}\quad\mathfrak{L}(\xi)(y):=\sum_{g(z)=y}\xi(z).
\end{equation}
We equip $E$ with the left action defined from the pointwise action. To emphasize this dependence, we write $E_{g}=\,\!_{\id} C(V)_{g^{*}}$. For $k\geq 0$, there is a unitary isomorphism
$$\nu_k:E^{\otimes k}_{g}=E_g\!\otimes E_g\!\otimes\cdots \!\otimes E_g\to \,\!_{\id}C(V)_{g^{*k}}, \quad \eta_1\otimes \cdots \otimes \eta_k\mapsto \prod_{j=1}^kg^{*(j-1)}(\eta_j).$$
Here $\,\!_{\id}C(V)_{g^{*k}}$ is equipped with the inner product $\langle \xi,\eta\rangle_{\,\!_{\id}C(V)_{g^{*k}}}:=\mathfrak{L}^k(\bar{\xi}\eta)$. We define $E^{\otimes 0}:=\,\!_{\id} C(V)_{\id}$. An element $\xi\in E$ induces a
\emph{Toeplitz operator} 
$$T_\xi:E^{\otimes k}\to E^{\otimes k+1},\quad \eta_1\otimes \cdots \otimes \eta_k\mapsto \xi\otimes \eta_1\otimes \cdots \otimes \eta_k.$$
The following computation is immediate:
$$T_\xi^* \eta_1\otimes \cdots \otimes \eta_k=\begin{cases}
\langle \xi,\eta_1\rangle_E\ \eta_2\otimes \cdots \otimes \eta_k,\;&k>1,\\
\langle \xi,\eta_1\rangle_E,&k=1,\\
0,& k=0.\end{cases}$$
Under the collection of isomorphisms $\nu_k$ we have
$$\tilde{T}_\xi:=\nu_{k+1}T_\xi\nu_k^{-1}:\,\!_{\id}C(V)_{g^{*k}}\to \,\!_{\id}C(V)_{g^{*(k+1)}},\quad \eta\mapsto \xi g^{*}(\eta),$$
and for $\eta\in \,\!_{\id}C(V)_{g^{*k}}$ we have 
$$\tilde{T}_\xi^* \eta=\begin{cases}
\mathfrak{L}(\bar{\xi}\eta)\;&k>0,\\
0,& k=0.\end{cases}$$

The \emph{Fock module} of $E$ is the $C(V)$-Hilbert module $F_E:=\bigoplus_{k=0}^\infty  \,\!_{\id}C(V)_{g^{*k}}$. The \emph{Toeplitz algebra} $\mathcal{T}_E\subseteq \End^*_{C(V)}(F_E)$ is the $C^*$-algebra generated by the Toeplitz operators $\{\tilde{T}_\xi:\; \xi\in E\}$ and the compact operators $\mathbb{K}_{C(V)}(F_E)$. The \emph{Cuntz-Pimsner algebra} of $E$ is
\[
O_E:=\mathcal{T}_E/\mathbb{K}_{C(V)}(F_E).
\]
If we equip $E$ with the $\T$-action $z\cdot \xi:=z\xi$, there is an induced $\T$-action on $F_E$ making $F_E$ into a $\T$-equivariant $C(V)$-Hilbert module for the trivial action on $C(V)$. Then $\mathcal{T}_E$ and $\mathbb{K}_{C(V)}(F_E)$ are invariant under the adjoint action. We equip these $C^*$-algebras and the Cuntz-Pimsner algebra $O_E$ with the $\T$-action induced from the adjoint action.

\begin{example}
\label{firstckex}
In the special case of a subshift of finite type, the construction above reproduces the associated Cuntz-Krieger algebra. Let $\pmb{A}$ be a $N\times N$-matrix consisting of zeros and ones. We consider the compact space 
$$\Omega_{\pmb{A}}:=\left\{(s_i)_{i\in \N}\in \{1,...,N\}^\N: \pmb{A}_{x_i,x_{i+1}}=1\;\forall i\right\},$$
which is equipped with the topology induced from the product topology. If $\pmb{A}$ is irreducible, then $\Omega_{\pmb{A}}$ is a Cantor set. The mapping $\sigma(s_i)_{i\in \N}:=(s_{i+1})_{i\in \N}$ is a surjective local homeomorphism. If we construct the $C(\Omega_{\pmb{A}})$-module $E$ from $(\Omega_{\pmb{A}},\sigma)$ we obtain the Cuntz-Krieger algebra defined from the matrix $\pmb{A}$. To prove this, consider the elements $S_i:=\tilde{T}_{\chi_{C_i}}\!\!\!\!\mod \mathbb{K}_{C(\Omega_{\pmb{A}})}$ where $C_i:=\{(s_i)_{i\in \N}\in \Omega_{\pmb{A}}: s_0=i\}$ is the clopen cylinder set on words starting with $i$. A direct computation gives
$$S_i^*S_k=\delta_{ik}\mathfrak{L}(\chi_{C_i})=\delta_{ik}\sum_{j=1}^N \pmb{A}_{ij}\chi_{C_j}\quad \mbox{and}\quad S_jS_j^*=\chi_{C_j}.$$
Hence $\{S_i\}_{i=1}^N$ satisfies the Cuntz-Krieger relations defined from $\pmb{A}$. This defines the isomorphism $O_{\pmb{A}}\to O_E$.
\end{example}

\subsection{The Cuntz-Pimsner algebra of a topological graph}
\label{cptopgraph}

In this section we consider the Cuntz-Pimsner algebras of topological graphs, and make use of an \'{e}tale groupoid previously considered in \cite{deaacaneudoaod,Renpeper,Ren2book}. Suppose $V$ is a compact topological space and consider a closed subset $\mathcal{G}\subseteq V\times V$. We set $t(x,y)=y$ and $o(x,y)=x$ for $(x,y)\in \mathcal{G}$. This situation is a special case of the notion of a topological graph (see \cite[Definition 1.1]{deaacaneudoaod}). Consider the one-sided sequence space
$$X_+(\mathcal{G}):=\left\{(x_i)_{i\in \N}\in \prod_{i\in \N}V: (x_i,x_{i+1})\in \mathcal{G} \,\, \quad\forall i\in\N\right\}.$$
There is a shift mapping $\sigma_\mathcal{G}:X_+(\mathcal{G})\to X_+(\mathcal{G})$, $\sigma_+(x_i)_{i\in \N}:=(x_{i+1})_{i\in \N}$. We define the set
$$\mathcal{R}_\mathcal{G}:=\left\{(x,n,y)\in X_+(\mathcal{G}) \times \Z\times X_+(\mathcal{G}): \exists k \mbox{  with  } \sigma_\mathcal{G}^{n+k}(x)=\sigma_\mathcal{G}^k(y)\right\}.$$
We can make $\mathcal{R}_\mathcal{G}$ into a groupoid by defining 
\begin{align}
r(x,n,y):=x,\quad d(x,n,y):=y\quad\mbox{and}\quad 
\label{compositionrg}
(x,n,y)(y,m,z):=(x,n+m,z).
\end{align} 
With the additional assumption that $t, o:\mathcal{G}\to V$ are surjective local homeomorphisms, the groupoid $\mathcal{R}_\mathcal{G}$ is topologized by the following basis. For $k,l\in \N$ and open subsets $U_1,U_2\subseteq X_+(\mathcal{G})$ such that $\sigma_+^k|_{U_1}$ and $\sigma_+^l|_{U_2}$ are homeomorphisms with the same open range, we declare the following set open
\begin{equation}
\label{basisgg}
\mathcal{U}(U_1,k,l,U_2):=\{(x,k-l,y)\in U_1\times\Z\times U_2: \sigma_+^k(x)=\sigma_+^l(y)\}.
\end{equation}
This construction makes $\mathcal{R}_\mathcal{G}$ into an \'{e}tale groupoid. Henceforth we will always assume that both $t$ and $o$ are surjective local homeomorphisms. 

Consider the $C(V)$-bimodule $C(\mathcal{G})$ with left and right action defined by $o^*$ and $t^*$, respectively. There is a transfer operator $\mathfrak{L}_\mathcal{G}:C(\mathcal{G})\to C(V)$ defined by $$\mathfrak{L}_\mathcal{G}(\xi)(y):=\sum_{(x,y)\in \mathcal{G}}\xi(x,y).$$ We equip $C(\mathcal{G})$ with the inner product $\langle\xi,\eta\rangle_{C(\mathcal{G})}:=\mathfrak{L}_\mathcal{G}(\bar{\xi}\eta)$. Then \cite[Proposition 3.3]{deaacaneudoaod} proves that
\[
O_{C(\mathcal{G})}\cong C^*(\mathcal{R}_\mathcal{G}),
\]
whenever $t$ and $o$ are surjective local homeomorphisms. This isomorphism is $\T$-equivariant for the $\T$-action on $C^*(\mathcal{R}_\mathcal{G})$ induced from the groupoid cocycle $c_\mathcal{G}(x,n,y):=n$.

The situation from \cite{deaacaneudoaod} fits into the theme of this paper through the graph
\begin{equation}
\label{thegraph}
\mathcal{G}_g:=\{(x,g(x))\in V\times V: x\in V\},
\end{equation}
for the surjective local homeomorphism $g:V\to V$ considered in the previous subsection. This satisfies all the conditions above and it is routine to shown that $E\cong C(\mathcal{G})$ as Hilbert bimodules using the pullback along $o:\mathcal{G}\to V$. It is immediate from the definition that there is a conjugacy between $(X_+(\mathcal{G}_g),\sigma_{\mathcal{G}_g})$ and $(V,g)$, and that
\begin{equation}
\label{definitiongroup}
\mathcal{R}_{\mathcal{G}_g}\cong \G_g:= \{(x,n,y)\in V \times \Z\times V: \exists k \mbox{  with  } g^{n+k}(x)=g^k(y)\}.
\end{equation}
We summarize the discussion above into the following result.

\begin{thm}
\label{isothemggtooh}
Suppose $g:V\to V$ is a surjective local homeomorphism defining a $*$-monomorphism $g^{*}$ on $C(V)$. Let $E_{g}:=\,\!_{\id} C(V)_{g^{*}}$, and let $\G_g$ be as in Equation \eqref{definitiongroup}.  The $*$-homomorphism $C(V)\to C^*(\G_g)$ defined from the diagonal inclusion and the mapping
$$t:E\to C^*(\G_g), \quad t(\xi)(x,n,y):=\begin{cases} \xi(x), \; &g(x)=y, \,n=1\\ 0,&\mbox{otherwise}\end{cases}$$
define a covariant representation that induces a $\T$-equivariant isomorphism
$$\pi_E:O_E \to C^*(\G_g).$$
\end{thm}

\begin{remark}
\label{reversedgraph}
The graph $\mathcal{G}_g^{\op}=\{(g(x),x)\in V\times V: x\in V\}$ gives rise to a dynamics very different from that of $\mathcal{G}_g$. Using \eqref{solenodiconst}, we have that $X_+(\mathcal{G}_g^{\op})=X_V$. The $C(V)$-Hilbert $C^*$-module $C(\mathcal{G}_g^{\op})$ is unitarily equivalent to $E^{\op}:=\;\!_\varphi C(V)_{\id}$ with the inner product $\langle \xi,\eta\rangle_{E^{\op}}=\bar{\xi}\eta$. Moreover, the identification $X_+(\mathcal{G}_g^{\op})=X_V$ induces a conjugacy 
\[
(X_+(\mathcal{G}^{\op}_g),\sigma_{\mathcal{G}_g^{\op}})\cong (X_V,\varphi_g^{-1}).
\]
That is, $\sigma_{\mathcal{G}_g^{\op}}$ is a homeomorphism. Hence, $\mathcal{R}_{\mathcal{G}_g^{\op}}\cong X_V \rtimes_{\varphi} \Z$ as groupoids, and by \cite[Proposition 3.3]{deaacaneudoaod} there is a $\T$-equivariant isomorphism $O_{E^{\op}}\cong C(X_V)\rtimes_{\varphi} \Z$ when equipping $C(X_V)\rtimes_{\varphi} \Z$ with the dual $\T$-action.
\end{remark}

\subsection{Wieler's axioms and Cuntz-Pimsner algebras}
\label{discussionong}

Let us discuss the assumption of $g$ being a local homeomorphism. For the purposes of the present paper, the minimal assumptions on $g$ should not only ensure that $X_{V}$ is a Smale space but also that there is an associated Cuntz-Pimsner algebra. In view of Theorem \ref{isothemggtooh}, to use $E:=C(V)_{g^*}$ it is necessary that $E$ admits a right Hilbert $C^*$-module structure. That is, the existence of a right inner product with values in $C(V)$ that is compatible with $g$. The question of existence of such inner products was considered in detail by Pavlov-Troitsky \cite{pavlovstrots} as follows.

\begin{cor}[{\cite[Theorem 1.1 and 2.9]{pavlovstrots}}]
Let $g:V\to V$ be a finite-to-one surjection. Then, the following are equivalent:
\begin{enumerate}
\item $E=C(V)_{g^*}$ admits a right $C(V)$-Hilbert $C^*$-module structure,
\item $g$ is a branched covering in the sense of \cite[Definition 2.4]{pavlovstrots}, and
\item $g$ is open.
\end{enumerate}
\end{cor}

Thus, we obtain a Cuntz-Pimsner model for $C^{*}(\G_{g})$ using $E$ only when $g:V\to V$ is an open map. Since $g$ also satisfies Wieler's axioms, we obtain the following result using Remark \ref{clobra}.

\begin{prop}
Let $g:V\to V$ be a surjection satisfying Wieler's axioms. Then, the following are equivalent:
\begin{enumerate}
\item $E=C(V)_{g^*}$ admits a right $C(V)$-Hilbert $C^*$-module structure,
\item $g$ is a local homeomorphism, and
\item $g$ is open.
\end{enumerate}
\end{prop} 

Thus, the assumption that $g$ is a local homeomorphism covers the study of Wieler solenoids that admit a Cuntz-Pimsner model of the kind described in Theorem \ref{isothemggtooh}.

\section{The Cuntz-Pimsner extension as an unbounded Kasparov module}
\label{sectionfour}

As in Section \ref{sectioncpmodel}, we consider a surjective local homeomorphism $g:V\to V$. Again, the  results in this section do {\bf not} assume that $g$ satisfies Wieler's axioms. In this section we study an explicit unbounded representative for the boundary mapping in $KK_1^{\T}(O_E,C(V))$ coming from the $\T$-equivariant short exact sequence
\begin{equation}
\label{sesforoh}
0\to \mathbb{K}_{C(V)}\to \mathcal{T}_E\to O_E\to 0.
\end{equation}
The consequences in $K$-theory and $K$-homology will be studied in Subsection \ref{someconseq}.

\subsection{The $\kappa$-function}
For a subshift of finite type, the groupoid $\G_{g}$ encodes the relation of shift-tail equivalence. The continuous cocycle $c(x,n,y):=n$ allows for a decomposition of $\G_{g}$ into clopen subsets. In \cite[Lemma 5.1.1]{goffmes} it was observed that this decomposition can be further refined using the natural number $k$ arising in the analogue of \eqref{definitiongroup}. We now show that a decomposition of the groupoid $\G_{g}$, as described in \cite{goffmes}, exists for a general local homeomorphism $g:V\to V$.

\begin{definition}
\label{cgg}
For a surjective local homeomorphism $g:V\to V$, we define $\kappa:\mathcal{R}_g\to \Z$ by 
$$\kappa(x,n,y):=\min\{k\in \N: \; g^{n+k}(x)=g^k(y)\}.$$
We also define $c_{\G_g}:\mathcal{R}_g\to \Z$ by 
\[
c_{\G_g}(x,n,y):=n.
\]
\end{definition}

It is implicitly understood in the definition of $\kappa$ that $k+n\geq 0$ so that both sides of the equation $g^{n+k}(x)=g^k(y)$ are well defined.

\begin{prop}
The functions $\kappa$ and $c_{\G_g}$ from Definition \ref{cgg} are locally constant and continuous.
\end{prop}

\begin{proof}
It suffices to prove that $c_{\G_g}$ and $\kappa$ are locally constant, then continuity follows automatically. On an open set of the form $\mathcal{U}(U_1,k,l,U_2)\subseteq \mathcal{R}_g$ from the basis of the topology in \eqref{basisgg}, we have that 
\[
\kappa|_{\mathcal{U}(U_1,k,l,U_2)}=l\quad\mbox{and}\quad c_{\G_g}|_{\mathcal{U}(U_1,k,l,U_2)}=k-l.
\]
Thus $\kappa$ and $c_{\G_g}$ are locally constant.
\end{proof}

We often suppress the index from $c_{\G_g}$. Since $\kappa$ is locally constant, we can decompose $\mathcal{R}_g$ into a disjoint union of clopen sets:
\begin{equation}
\label{kappacdec}
\mathcal{R}_g=\dot{\cup}_{n\in \Z}\dot{\cup}_{k\geq -n} \mathcal{R}_g^{n,k},\quad\mbox{where}\quad \mathcal{R}_g^{n,k}:=c^{-1}(\{n\})\cap \kappa^{-1}(\{k\}).
\end{equation}
Define the $C(V)$-Hilbert $C^*$-modules $\Xi_{n,k}:=C(\mathcal{R}_g^{n,k})\subseteq C_c(\mathcal{R}_g)$. We equip $\Xi_{n,k}$ with the $\T$-action defined from $c$; that is, $z\cdot \xi:=z^n\xi$ for $\xi\in \Xi_{n,k}$. Also let $L^2(\G_g)_{C(V)}$ denote the completion of $C_c(\G_g)$ as a $C(V)$-Hilbert $C^*$-module in the inner product defined from the expectation \begin{equation}
\label{Rexp}
\varrho:C_c(\G_g)\to C(V),\quad\varrho(f)(x):=f(x,0,x).
\end{equation}
We then equip $L^2(\G_g)_{C(V)}$ with the $\T$-action defined from $c$; that is, $(z\cdot f)(x,n,y):=z^nf(x,n,y)$.

\begin{remark}
The notation $\Xi_{n,k}$ aligns with the notation $\Xi_{n,r}$ used in \cite[Section 3.2]{GMR} upon identifying $k$ with $r-n$. 
\end{remark}

\begin{prop}
\label{fingenproandmore}
The modules $\Xi_{n,k}$ are finitely generated projective $C(V)$-modules. Moreover, the identification 
$$E^{\otimes n}\cong C(V\times_gV\times_g\cdots \times_g V)\cong C_c(\mathcal{R}_g^{n,0})=\Xi_{n,0}$$
defines a $\T$-equivariant unitary isomorphism $u_n:E^{\otimes n}\to \Xi_{n,0}$. 
\end{prop}

\begin{proof}
It follows from Lemma \ref{morelocaltriv} that $\mathcal{R}_g^{n,k}\to V$ is a finite-to-one covering map, so $\Xi_{n,k}$ is a finitely generated projective $C(V)$-module. That $u_n$ is a unitary operator follows from a short computation.
 \end{proof}

\begin{prop}
\label{directsumprop}
There is an orthogonal direct sum decomposition of $\T$-equivariant $C(V)$-Hilbert $C^*$-modules
\[
L^2(\mathcal{R}_g)_{C(V)}=\bigoplus_{n\in \Z}\ \bigoplus _{k\geq-n} \Xi_{n,k}.
\]
\end{prop}

\begin{proof}
It is immediate from Equation \eqref{kappacdec} that $L^2(\G_g)_{C(V)}$ coincides with the closed linear span of $\Xi_{n,k}$ as $n$ and $k$ varies. The construction of the inner product shows that elements of $C_c(\G_g)$ with disjoint supports are orthogonal, hence $\Xi_{n,k}\perp \Xi_{n',k'}$ if $n\neq n'$ or $k\neq k'$.
\end{proof}

\subsection{The unbounded representative}

The functions $c$ and $\kappa$ combine into an unbounded $KK$-cycle in the same way as in \cite{goffmes}. Let $T:=\{(n,k)\in \Z\times \N: \; k\geq -n\}$. We consider the function 
$$\psi:T\to \Z, \quad \psi(n,k):=
\begin{cases} 
n, \;&k=0,\\
-|n|-k, \;&k>0.\end{cases}$$

\begin{prop}\label{almostcycle}
The operator $D_0:=\psi(c,\kappa):C_c(\G_g)\to C_c(\G_g)$ defines a $\T$-equivariant self-adjoint regular operator $D:\textnormal{Dom} (D)\rightarrow L^2(\G_g)_{C(V)}$ with compact resolvent and spectrum $\Z$. The positive spectral projection of $D$ coincides with the projection onto $\bigoplus_{n=0}^\infty \Xi_{n,0}$.
\end{prop}

\begin{proof}
The operator $(i\pm D_0)^{-1}:C_c(\G_g)\to C_c(\G_g)$ is bijective and $D_{0}$ is symmetric for the inner product induced by the expectation \eqref{Rexp}, hence its closure $D$ is regular and self-adjoint. Moreover, we can write
$$(i\pm D)^{-1}=\sum_{n,k}(i\pm \psi(n,k))^{-1}p_{n,k},$$
where $p_{n,k}$ denotes the projection onto the finitely generated projective module $\Xi_{n,k}$. Since $\psi(n,k)\to \infty$ as $(n,k)\to \infty$ in $T$, it follows that $(i\pm D)^{-1}$ is the norm limit of finite rank operators. Moreover, $\psi(n,k)\geq 0$ if and only if $k=0$ and the last statement in the proposition follows.
\end{proof}

Recall the notation $\pi_E:O_E\to C^*(\G_g)$ for the isomorphism from Theorem \ref{isothemggtooh}. Let $\beta_E:O_E\to \mathcal{Q}(F_E):=\End^*_{C(V)}(F_E)/\mathbb{K}_{C(V)}(F_E)$ denote the $\T$-equivariant Busby invariant of the extension \eqref{sesforoh}. For a Hilbert $C^{*}$-module $E$, we let $$q:\End^*_{C(V)}(E)\to \mathcal{Q}(E):=\End^*_{C(V)}(E)/\mathbb{K}_{C(V)}(E)$$ denote the quotient mapping. Recall the definition of $u_n:E^{\otimes n}\to \Xi_{n,0}$ from Proposition \ref{fingenproandmore}.

\begin{prop}
\label{uniteqbus}
The $\T$-equivariant inner product preserving adjointable mapping $u:=\oplus_{n\in \N}u_n:F_E\to \bigoplus_{n=0}^\infty \Xi_{n,0}\hookrightarrow L^2(\mathcal{R}_g)_{C(V)}$ satisfies 
$$\beta_E=q\circ Ad(u)\circ \pi_E:O_E\to\mathcal{Q}(F_E).$$
\end{prop}

\begin{proof}
It suffices to prove $\beta_E(a)=q\circ Ad(u)\circ \pi_E(a)$ for $a\in C(V)\cup \{\tilde{T}_\xi\mod \mathbb{K}: \;\xi\in E\}$ because this set generates $O_E$. The equation is trivially satisfied for $a\in C(V)$. For $\xi\in E$, Theorem \ref{isothemggtooh} shows that $\pi_E(\tilde{T}_\xi)=t_\xi$. The element $\eta\in E^{\otimes n}=C(V\times_g\cdots \times _g V)$ is mapped by $u_n$ to the element 
$$u\eta(x,n,y)=\eta(x,g(x),\ldots, g^{n-1}(x),y).$$
It holds that
$$t_\xi u\eta (x,n+1,y)=\xi(x)\eta(g(x),g^2(x),\ldots, g^{n}(x),y)=[u(\xi\otimes_{C(Y)} \eta)](x,n+1,y).$$
Therefore $u^*t_\xi u=\tilde{T}_\xi$ and $\beta_E(\tilde{T}_\xi)=q\circ Ad(u)\circ \pi_E(\tilde{T}_\xi)$
\end{proof}

\begin{thm}
\label{represent!}
Let $g:V\to V$ be a surjective local homeomorphism on a compact space $V$. The $\T$-equivariant unbounded $KK_1$-cycle $(L^2(\G_g)_{C(V)},D)$ for $(C^*(\G_g),C(V))$ represents the extension \eqref{sesforoh} in $KK_1^{\T}(O_E,C(V))$ under the isomorphism $\pi_E$.
\end{thm}

\begin{proof} 
By Proposition \ref{almostcycle}, $(L^2(\G_g)_{C(V)},D)$ is a $\T$-equivariant unbounded Kasparov module if $D$ has bounded commutators with a dense subalgebra of $C^{*}(\mathcal{R}_{g})$. We proceed to show that for $f\in C(\mathcal{R}_{g}^{1,0})$ the commutator
\[[D,f]:C_{c}(\mathcal{R}_{g})\to C_{c}(\mathcal{R}_{g}),\]
extends to a bounded operator on $L^{2}(\mathcal{R}_{g})_{C(V)}$. Since $C_{c}(\mathcal{R}_{g})$ is a core for $D$ and $C(\mathcal{R}^{1,0}_{g})$ generates $C_{c}(\mathcal{R}_{g})$ as a $*$-algebra this suffices. 

We compute
\begin{align*}[D,f]h(x,n,y)&=(D(f*h)- f*Dh)(x,n,y)\\
&=\sum_{(x,m,z)\in \supp f}\big[\psi(x,n,y)f(x,m,z)h(z,n-m,y)\\ 
&\quad\quad\quad\quad\quad\quad\quad\quad\quad - f(x,m,z)\psi(z,n-m,y)h(z,n-m,y)\big]\\
&=(\psi(x,n,y)-\psi(g(x),n-1,y))f(x,1,g(x))h(g(x),n-1,y)\\
&=(\psi(x,n,y)-\psi(g(x),n-1,y))(f*h)(x,n,y).
\end{align*}
A standard computation shows that pointwise multiplication by a bounded function in $C_{b}(\mathcal{R}_{g})$ defines an adjointable operator on $L^{2}(\mathcal{R}_{g})_{C(V)}$. Thus it suffices to establish that 
\begin{equation}
\label{starfunction}
(x,n,y)\mapsto \psi(x,n,y)-\psi(g(x),n-1,y)
\end{equation}
is a bounded function. This follows from distinguishing the following four cases:
\begin{enumerate}
\item $\kappa(x,n,y)=\kappa(g(x),n-1,y)=0$: In this case $\psi(x,n,y)-\psi(g(x),n-1,y)=n-(n-1)=1$;
\item $0=\kappa(x,n,y)<\kappa(g(x),n-1,y)$: $\kappa(x,n,y)=0$ gives that $n\geq 0$ and $g^{n}(x)=y$, whereas $k:=\kappa(g(x),n-1,y)>0$ means that $g^{k}(y)=g^{n-1+k}(g(x))=g^{n+k}(x)$. This contradicts the minimality of $\kappa(x,n,y)$ unless $n=0$, in which case it must hold that $\kappa(g(x),n-1,y)=1$. Then
\[\psi(x,n,y)-\psi(g(x),n-1,y)=n+|n-1|+\kappa(g(x),n-1,y)=2.\]
\item $0=\kappa(g(x),n-1,y)<\kappa(x,n,y)$: This case is void because $\kappa(g(x),n-1,y)=0$ gives that $n-1\geq 0$ and $g^{n}(x)=y$, whereas $k:=\kappa(x,n,y)>0$ means that $g^{n+k}(x)=g^{k}(y)$, which contradicts the minimality of $\kappa(x,n,y)$.
\item $\min\{\kappa(x,n,y),\kappa(g(x),n-1,y)\}>0$: In this case we compute
\begin{align*}
|\psi(x,n,y)-\psi(g(x),n-1,y)|&=|-|n|-\kappa(x,n,y)+|n-1|+\kappa(g(x),n-1,y)|\\ &\leq 1+|\kappa(x,n,y)-\kappa(g(x),n-1,y)|\leq 2.
\end{align*}
The last inequality follow from the observation that $n+\kappa(x,n,y)\geq 0$ and if $n+\kappa(x,n,y)>0$ then $\kappa(x,n,y)=\kappa(g(x),n-1 ,y)$ by minimality considerations. If $n+\kappa(x,n,y)=0$, then $n<0$ and it must hold that $\kappa(g(x),n-1,y)=\kappa(x,n,y)+1$.
\end{enumerate}
Therefore, $| \psi(x,n,y)-\psi(g(x), n-1, y)|\leq 2$ for all $(x,n,y)$, and \eqref{starfunction} defines a bounded function. 

One observes that Proposition \ref{uniteqbus} implies that the $\T$-equivariant Busby invariant of the extension \eqref{sesforoh} is unitarily equivalent to the Busby invariant associated with the $\T$-equivariant unbounded $KK^1$-cycle $(L^2(\G_g)_{C(V)},D)$ to $O_E$. 
\end{proof}

\begin{remark} 
In the case that $(V,g)$ is a subshift of finite type associated to a matrix $\pmb{A}$ and $\lambda$ is a finite $\pmb{A}$-admissible word, a family of unbounded cycles $(C^{*}(\G_{g}),L^{2}(\G_{g})_{C(V)}, D_{\lambda})$ was constructed in \cite[Theorem 5.1.7]{goffmes}. The cycle in Theorem \ref{represent!} recovers this construction for $\lambda=\circ$, the empty word. The proof that $[D,f]$ is bounded is verbatim the same as the proof of \cite[Lemma 5.1.6]{goffmes} for $\lambda=\circ$. Moreover, Theorem \ref{represent!} can also be obtained as a special case of \cite[Theorem 2.16]{GMR}. This is done by adapting the discussion in \cite[Section 3.2]{GMR} to the case of a general surjective local homeomorphism.
\end{remark}

\subsection{The Pimsner sequence and its consequences}
\label{someconseq}

The extension \eqref{sesforoh} plays an important r\^{o}le in the computation of the $K$-theory and $K$-homology of Cuntz-Pimsner algebras (see \cite{Pimsnerspaper, GMR}). The next theorem follows from Theorem \ref{represent!} and results in \cite[Section 4]{Pimsnerspaper}. The results in \cite{Pimsnerspaper} are formulated in the non-equivariant setting, but the proofs extend mutatis mutandis to the $\T$-equivariant setting.

\begin{thm}
\label{pimsnerthm}
Let $g:V\to V$ be a surjective local homeomorphism on a compact space $V$. There is an exact triangle in the triangulated category $KK^{\T}$:
\[\xymatrix@C=0.8em@R=2.71em{
C(V) \ar[rr]^{1-[E_g]}& & C(V) \ar[dl]^\iota\\
  &C^*(\G_g). \ar[ul]|-\circ^{[L^2(\G_g)_{C(V)},D]}
}\]
Here $\iota\in KK_0^{\T}(C(V),C^*(\G_g))$ is induced from the inclusion $C(V)\hookrightarrow C^*(\G_g)$ and $[E_g]\in KK_0^{\T}(C(V),C(V))$ from the bimodule $E_g$ equipped with the $\T$-action $z\cdot \xi:=z\xi$.
\end{thm}

In practice, the theorem above is most useful to deduce six term exact sequences in $KK$-theory.

\begin{cor}
\label{corollarpseqpage13}
For any separable $C^*$-algebra $B$, we have the following two six-term exact sequences:
\[\begin{CD}
KK_0(B,C(V))@>1-[E_g]>>KK_0(B,C(V)) @>\iota_*>> KK_0(B,C^*(\G_g))\\
@A-\otimes [L^2(\G_g)_{C(V)},D]AA  @. @VV-\otimes [L^2(\G_g)_{C(V)},D]V\\
KK_1(B,C^*(\G_g)) @<\iota_*<<KK_1(B,C(V))@<1-[E_g]<< KK_1(B,C(V))\\
\end{CD} \]
\[\begin{CD}
KK_0(C^*(\G_g),B)@>\iota^*>>KK_0(C(V),B)@>1-[E_g]>>KK_0(C(V),B) \\
@A[L^2(\G_g)_{C(V)},D]\otimes -AA  @. @VV[L^2(\G_g)_{C(V)},D]\otimes -V\\
KK_1(C(V),B) @<1-[E_g]<<KK_1(C(V),B)@<\iota^*<<KK_1(C^*(\G_g),B)\\
\end{CD} \]
The same statement holds in $\T$-equivariant $KK$-theory if $B$ is a separable $\T-C^*$-algebra.
\end{cor}

The next proposition is useful for computations with the sequences in Corollary \ref{corollarpseqpage13}; it is inspired by \cite[Section 3]{raeburnetal}. We first introduce some notations. We choose a cover $(U_j)_{j=1}^N$ of $V$ such that $g|_{U_j}$ is injective for any $j$. Choose a subordinate partition of unity $(\chi_j^{2})_{j=1}^N$. It is well-known that $\chi_{j}$ is a frame for $E_{g}$ and the right module mapping 
$$v:E_g\to C(V)^N, \quad x\mapsto (\langle \chi_j,x\rangle_E)_{j=1}^N,$$
is inner product preserving, i.e. $v^*v=\mathrm{id}_{E_g}$. More generally, whenever $(e_j)_{j=1}^N$ is a frame for $E_g$ we can define $v$ as above. We denote the associated left representation by 
\begin{equation}
\label{omegadefin}
\theta:C(V)\to C(V,M_N(\C)), \quad a\mapsto vav^*.
\end{equation}

\begin{prop}
\label{propcomposing}
Let $[\theta]\in KK_0^{\T}(C(V),C(V))$ denote the class associated with the equivariant $*$-homomorphism $\theta$ and $t\in KK_0^{\T}(\C,\C)\cong R(\T)\cong \Z[t,t^{-1}]$ the class associated with the representation given by the inclusion $\T\subseteq \C$. Then 
$$[E_g]=t\otimes_\C [\theta] \in KK_0^{\T}(C(V),C(V)).$$
Define ${}_gE_g$ as the bimodule ${}_{g^*}C(V)_{g^*}$, or equivalently the trivial bimodule structure associated with the right module structure on $E_g$. Then 
$$[E_g]\otimes_{C(V)} [g^*]=[{}_gE_g]\quad\mbox{  in  }KK_0^{\T}(C(V),C(V)).$$
In particular, if $E_g$ is free of rank $r$ as a right module, then $[E_g]\otimes_{C(V)} [g^*]=rt\otimes_\C[1_{C(V)}]$.
\end{prop}

\begin{proof}
Let $E_g^\perp$ denote the right module $(1-vv^*)C(V)^N$ equipped with the trivial left action of $C(V)$ and the trivial $\T$-action. Let ${}_\theta C(V)^N$ denote the right Hilbert module $C(V)^N$ equipped with the left action defined from $\theta$ and the trivial $\T$-action. The isometry $v$ and the inclusion $E_g^\perp\subseteq C(V)^N$ implements an equivariant unitary equivalence of bimodules $(t^{-1}\otimes E_g)\oplus E_g^\perp\cong {}_\theta C(V)^N$.  The second statement follows immediately from the first statement since $v$ is $C(V)$-linear, so $\theta(g^*(a))=vg^*(a)v^*=g(a)vv^*$. 
\end{proof}

\begin{remark}
\label{someremarksonpex}
The applicable aspect of Corollary \ref{corollarpseqpage13} is the computation of $K$-theoretic invariants from the knowledge of the action of $E_g$ on $KK$-theory. The non-equivariant version gives $K$-theoretic information about $C^*(\G_g)$. The equivariant version produces information about the fixed point algebra $C^*(\G_g)^{\T}$. This is of interest below for the stable and unstable algebras of Smale spaces, see the discussion in Remark \ref{uoneremark}. The $K$-theory of $C^*(\G_g)^{\T}$ is computed using the Green-Julg theorem and the Morita equivalence $C^*(\G_g)\rtimes \T\sim_MC^*(\G_g)^{\T}$ induces an isomorphism $K_*^{\T}(C^*(\G_g))\cong K_*(C^*(\G_g)^{\T})$. Moreover, if $C^*(\G_g)$ is $\T$-equivariantly Poincar\'e dual to a $\T-C^*$-algebra $D$, then $K_*(D\rtimes \T)\cong K^{*+j}_{\T}(C^*(\G_g))$ where $j$ is the dimension of the Poincar\'e duality. 
\end{remark}

\begin{example}
\label{computingktehroyexp}
Let us do some computations on the equivariant $K$-theory and $K$-homology of the solenoid associated with an expansive matrix $A\in M_d(\Z)$ from Example \ref{solenioidexampleonaroundpage7}. This computation extends \cite[Theorem 4.9]{raeburnetal}. We let $g_A:(S^1)^d\to (S^1)^d$ denote the mapping associated with the expansive integer matrix $A$. 

The module $E_{g_A}$ is free of rank $|\det(A)|$ as a right $C((S^1)^d)$-module (see \cite[Lemma 2.6]{raeburnetal}). By Proposition \ref{propcomposing}, 
$$[E_{g_A}]\otimes_{C((S^1)^d)} [g^*_A]=|\det(A)|t\otimes_\C[1_{C((S^1)^d)}]\quad\mbox{in $KK^{\T}(C((S^1)^d),C((S^1)^d))$}.$$ 
We emphasize that the $\T$-action on $C((S^1)^d)$ is trivial and $t$ is the generator of the representation ring $R(\T)$. The equivariant $K$-theory group $K_*^{\T}(C((S^1)^d))\cong \wedge^*\Z^d\otimes \Z[t,t^{-1}]$ can be computed as a module over $R(\T)\cong \Z[t,t^{-1}]$. One easily verifies that $[g_A^*]$ acts as $\wedge^*A\otimes \id_{\Z[t,t^{-1}]}$ under this isomorphism, so $[E_{g_A}]$ acts as $B\otimes t$ where 
\[
B=\bigoplus_{j=0}^d B_j\in \bigoplus_{j=0}^d\End(\wedge^j\Z^d),
\]
 satisfies $B_j\wedge^jA=N\id_{\wedge^j\Z^d}$. The matrix $B$ is computed in \cite[Proposition 4.6]{raeburnetal}. From these considerations and Corollary \ref{corollarpseqpage13}, we then have
\begin{equation}
\label{equicompsole}
K_i^{\T}(O_{E_{g_A}})\cong \bigoplus_{j\in i+2\Z} \mathrm{coker}\left(1-B_{j}t:\wedge^{j}\Z^d\otimes \Z[t,t^{-1}]\to\wedge^{j}\Z^d\otimes \Z[t,t^{-1}] \right).
\end{equation}
This gives a new proof of the computation of $K_*(O_{E_{g_A}})$ in \cite[Theorem 4.9]{raeburnetal} using the Pimsner-Voiculescu sequence, i.e. a localization in the trivial $\T$-representation. For instance, if $d=1$ and $A=n\in \Z\setminus \{-1,0,1\}$, then $K_0^\T(O_{E_{g_A}})=\Z\left[\frac{1}{n}\right]$ and $K_1^\T(O_{E_{g_A}})=\Z$ with $t$ acting as $\frac{1}{|n|}$ in degree $0$ and as $\mathrm{sign}(n)$ in degree $1$.

Similarly, using the fact that $1-B_j^T$ is injective for $0\leq j<d$ (see \cite[Proposition 4.6]{raeburnetal}) and Corollary \ref{corollarpseqpage13}, we compute that 
$$K^i(O_{E_{g_A}})\cong 
\begin{cases}
\Z\oplus \bigoplus_{j<d, \;\mathrm{odd}} \mathrm{coker}(1-B_j^T), \quad &\mbox{for $i$ even and $\det(A)>1$},\\
\Z\oplus \bigoplus_{j<d, \;\mathrm{even}} \mathrm{coker}(1-B_j^T), \quad &\mbox{for $i$ odd and $\det(A)>1$},\\
\bigoplus_{j\leq d, \;\mathrm{odd}} \mathrm{coker}(1-B_j^T), \quad &\mbox{for $i$ even and $\det(A)<-1$},\\
\bigoplus_{j\leq  d, \;\mathrm{even}} \mathrm{coker}(1-B_j^T), \quad &\mbox{for $i$ odd and $\det(A)<-1$}.
\end{cases}$$

\end{example}

\begin{cor}
\label{corollarofpseqonpage14}
Suppose that $x\in KK_*(C^*(\G_g),B)$. The equation 
$$x=[L^2(\G_g)_{C(V)},D]\otimes_{C(V)} y, \quad y\in KK_{*-1}(C(V),B),$$
admits a solution $y$ if and only if $\iota^*(x)=0$ holds in $KK_*(C(V),B)$. The analogous $\T$-equivariant statement holds as well.
\end{cor}

\section{The Morita equivalence with the stable Ruelle algebra}
\label{sectionfive}

The standing assumption in this section is that $(V,g)$ is a pair satisfying Wieler's axioms such that the associated Smale space $(X,\varphi)$ is irreducible. Suppose $P$ is a finite $\varphi$-invariant set of periodic points of $(X,\varphi)$. Recall from Section \ref{subsec:Algebras} the stable groupoid $G^s(P)$, the stable algebra $C^*(G^s(P))$ and the stable Ruelle algebra $C^*(G^s(P)) \rtimes \Z$ constructed from $(X,\varphi)$. We will define a groupoid Morita equivalence $G^s(P)\rtimes \Z\sim \G_g$ that will give rise to a Morita equivalence $C^*(G^s(P)) \rtimes \Z \sim_M O_E$ using Theorem \ref{isothemggtooh}. We then proceed to compute the product of the cycle in Theorem \ref{represent!} with the Morita equivalence, using a $\kappa$-type function as in Section \ref{sectionfour}.

\subsection{The topological space $Z^{u}(P)$} The \'{e}tale groupoid $G^{s}(P)$ carries a continuous action of the integers $\Z$ defined by $\alpha^{n}:(x,y)\mapsto (\varphi^{n}(x),\varphi^{n}(y))$ (the map $\alpha$ was defined in Section \ref{subsec:Algebras}). Recall the construction of crossed products of a groupoid with a group action. Assume that $G$ is an \'{e}tale groupoid with a right action of a discrete group $\Gamma$. We can form the crossed product groupoid $G\rtimes \Gamma$ with unit space $G^{(0)}$ by setting $G\rtimes \Gamma:= G\times \Gamma$ with domain, range and inverse mappings
\[
d_{G\rtimes \Gamma}(g,\gamma)=d_{G}(g\gamma), \quad\quad r_{G\rtimes \Gamma}(g,\gamma):=r_{G}(g)\quad \text{ and }\quad (g,\gamma)^{-1}:=(g\gamma,\gamma^{-1}).
\]
The composition is defined whenever $d(g_{1}\gamma_{1})=r(g_{2})$ and is given by
$$(g_{1},\gamma_1)\cdot (g_2,\gamma_2)=(g_1(g_2\gamma_{1}^{-1}), \gamma_1\gamma_2).$$

\begin{notation}
For notational convenience and compatibility with $\G_g$ (cf. \eqref{compositionrg}),  we identify the groupoid $G^{s}(P)\rtimes\Z$ with the set of triples $(x,n,y)$ with $(x,\varphi^{-n}(y))\in G^{s}(P)$ and $n\in\Z$. The range and domain maps in this model of $G^{s}(P)\rtimes\Z$ are given by $r(x,n,y):=x$ and $d(x,n,y):=y$, respectively. The composition is defined by $(x,n,y)(y,m,z)=(x,n+m,z)$. 
\end{notation}

\begin{definition}
\label{cgspz}
Similarly to Definition \ref{cgg}, we define the groupoid cocycle 
\[
c_{G^{s}(P)\rtimes\Z}:G^{s}(P)\rtimes\Z\to \Z, \quad c_{G^{s}(P)\rtimes\Z}(x,n,y):=n.
\]
\end{definition}

The $\Z$-grading on $C^*(G^{s}(P)\rtimes\Z)$ defined from $c_{G^{s}(P)\rtimes\Z}$ coincides with the grading coming from the crossed product structure $C^*(G^{s}(P)\rtimes\Z)\cong C^*(G^{s}(P))\rtimes\Z$. 

We define the space
\[
Z:=\{(x,j,v)\in X\times \Z\times V: \exists k\mbox{  such that  }  g^{k+j}(\pi_{0}(x))=g^{k}(v)\},
\]
and consider the subspace 
\begin{equation}
\label{zupdefinition}
Z^u(P):=\{(x,j,v)\in Z: \, x\in X^u(P)\} \subset Z.
\end{equation}
We will provide a locally compact Hausdorff topology on $Z^u(P)$ below.

\begin{lemma}
\label{triviality} 
If $(x,j,v)\in Z^u(P)$ and $y \in X$ is such that $\pi_0(y)=v$, then $y \sim_s \varphi^{j}(x)$.
\end{lemma}

\begin{proof} 
Write $x=(x_{0},x_{1},\cdots)$ and observe that by definition of $Z^{u}(P)$ there exists $k\in \N$ such that $g^{k+j}(x_{0})=g^{k}(v)$. For $y_{1}, y_{2},\cdots \in V$ such that $y:=(v,y_{1},\cdots )\in X$, we have 
$$\varphi^{k}(y)=(g^{k}(v),g^{k-1}(v),\cdots, v,y_{1},\cdots)=(g^{k+j}(x_{0}),\cdots).$$ 
Therefore, $\varphi^{k}(y)_{0}=\varphi^{k+j}(x)_{0}$. Using \cite[Observation preceding Lemma 3.1]{Wielerpaper} and a short induction argument we find that $\d_X(\varphi^{n+j}(x),\varphi^{n}(y))=\gamma^{n-k}\d_X(\varphi^{k+j}(x),\varphi^{k}(y))$ for $n\geq k$. In particular, we deduce 
$$\lim_{n\to\infty}\d_X(\varphi^{n+j}(x),\varphi^{n}(y))=0,$$ 
so that $\varphi^{j}(x)\sim_{s} y$.
\end{proof}

To define a topology on $Z^u(P)$ we construct a neighbourhood basis of $(x,j,v)\in Z^{u}(P)$ as follows. By Lemma \ref{triviality}, $(\varphi^{j}(x),y)\in G^{s}(P)$ for any $y\in X^{u}(P)$ with $\pi_{0}(y)=v$. Such a $y$ always exists by Theorem  \ref{localfibb}. Take a local neighbourhood $V^s(\varphi^{j}(x),y,N,U)\subseteq G^s(P)$ of $(\varphi^{j}(x),y)$ as in Equation \eqref{dahood}, with the following additional requirements:
\begin{enumerate}
\item $\pi_{0}:U\to \pi_{0}(U)$ is a homeomorphism;
\item $g^{N}$ is injective on $\pi_{0}(U)$;
\item $\pi_{0}: \varphi^{-j}(h_{N}^{s}(U))\to \pi_{0}(\varphi^{-j}(h_{N}^{s}(U)))$ is a homeomorphism;
\item $g^{N+j}$ is injective on $\pi_{0}(\varphi^{-j}h^{s}_{N}(U))$.
\end{enumerate}
Existence of sufficiently many such sets $U$ also follows from Theorem \ref{localfibb}. We declare the following set to be an element in the basis of the topology of $Z^{u}(P)$:
\small
\begin{equation}
\label{basiszup}
W^{s}((x,j,v), N, U):=
\{(\varphi^{-j}(x'),j,\pi_0(y'))\in Z^{u}(P): \; (x',y')\in V^s(\varphi^{j}(x),y,N,U)\}.
\end{equation}
\normalsize
Note that $(\varphi^{-j}(x'),j,\pi_0(y')) \in Z^u(P)$ for all $(x',y')\in V^s(\varphi^{j}(x),y,N,U)$ by the following argument. We have $[\varphi^{N}(y'),\varphi^{N}(x)]\in X^{s}(\varphi^{N}(y'),\ep_{X}/2)$. By Lemma \ref{bracketEqualsStableForSmallN}, it follows that 
$$\pi_{0}([\varphi^{N}(y'),\varphi^{N}(x)])=\pi_{0}(\varphi^{N}(y')).$$ 
Then, $g^{N+j}(\pi_{0}(\varphi^{-j}(h_{N}^{s}(y')))=\pi_{0}([\varphi^{N}(y'),\varphi^{N}(x)])=\pi_{0}(\varphi^{N}(y'))=g^{N}(\pi_{0}(y'))$.

\begin{prop}
\label{momentmaps}
The maps $\rho_L:Z^u(P)\to X^u(P)$, $(x,j,v)\mapsto x$, $\rho_R:Z^u(P)\to V$, $(x,j,v)\mapsto v$, and $\pi_{0}:Z^{u}(P)\to \G_{g}$, $(x,j,v)\mapsto (\pi_{0}(x),j,v)$ are surjective local homeomorphisms.
\end{prop}

\begin{proof} 
The statement for $\rho_{L}$ follows directly from the definition of the local neighbourhoods of $Z^{u}(P)$. For $\rho_{R}$ we know by Theorem \ref{localfibb} that $\pi_0:X^u(P)\to V$ is a covering map. Thus for $v\in V$ there exists $x\in X^{u}(P)$ with $\pi_{0}(x)=v$ and $(x,0,v)\in Z^{u}(P)$. Then $v=\rho_{R}(x,0,v)$ so $\rho_{R}$ is surjective too. For a local neighbourhood $W^{s}((x,j,v),N,U)$ with $U$ sufficiently small, we have that the open set $\rho_{R}(W^{s}((x,j,v),N,U))=\pi_{0}(U)$ is homeomorphic to $U$ because of Theorem \ref{localfibb}. For the map $\pi_{0}$, the local neighbourhoods in $Z^{u}(P)$ are defined in such a way that $W^{s}((x,j,v),N,U)$ is homeomorphic to $\pi_{0}(W^{s}((x,j,v),N,U))$, which proves the statement.
\end{proof}

We now establish that the groupoids $G^{s}(P)\rtimes \Z$ and $\mathcal{R}_{g}$ are Morita equivalent in the sense of \cite{MRW}. This implies that the groupoid $C^{*}$-algebras $C^{*}(\mathcal{R}_{g})$ and $C^{*}(G^{s}(P)\rtimes \Z)$ are strongly Morita equivalent. An equivalence of groupoids is implemented by a topological space carrying appropriate left and right actions. We now show that the space $Z^{u}(P)$ constructed above implements the desired Morita equivalence.

\begin{prop}
\label{Moritaprop}
The map 
\[
\zeta:Z^u(P)\to X^{u}(P)_{\pi_0}\times_r\G_{g}:=\{(x,g)\in X^{u}(P)\times\G_{g}:\pi_0(x)=r(g)\}
\]
given by $(x,j,v) \mapsto (x,(\pi_0(x),j,v))$ is a homeomorphism,
where the right hand side is equipped with the topology from the inclusion $X^{u}(P)_{\pi_0}\times_r\G_{g}\subseteq X^{u}(P)\times\G_{g}$. In particular, the actions 
\begin{align*}
(G^s(P)\rtimes \Z)_d\times_{\rho_L} Z^u(P)&\to Z^u(P),\quad (x,n,y)\cdot(y,j,v)= (x,j+n,v),\\
Z^u(P)_{\rho_R}\times_r \G_g&\to Z^u(P),\quad (x,j,v)\cdot (v,n,w)= (x,j+n,w),
\end{align*}
are well defined, free and proper.
\end{prop}

The topology of $Z^u(P)$ is constructed in such a way that $\zeta$ is a homeomorphism; the inverse of $\zeta$ is the continuous mapping $\zeta^{-1}(x,(\pi_0(x),j,v))=(x,j,v)$. That the actions are free and proper follows from a lengthier computation which we omit. We deduce the following theorem from Proposition \ref{Moritaprop}. 

\begin{thm}
\label{moreth}
Let $(X,\varphi)$ be an irreducible Wieler solenoid arising from an open surjection $g:V \to V$ and suppose $P\subseteq X$ is a finite $\varphi$-invariant set of periodic points.  The space $Z^u(P)$ defines an \'{e}tale groupoid Morita equivalence
\[
G^s(P)\rtimes \Z\sim_M \G_g,
\]
that respects the cocycles $c_{G^s(P)\rtimes \Z}$ (see Definition \ref{cgspz}) and $c_{\G_g}$ (see Definition \ref{cgg}). In particular, there is a $\T$-equivariant Morita equivalence $C^*(G^s(P))\rtimes \Z\sim_M O_E$.
\end{thm}

\begin{remark}
\label{remarkonmodules}
The $(C^{*}(G^{s}(P))\rtimes \Z, C^{*}(\G_{g}))$ Morita equivalence bimodule $L^{2}(Z^{u}(P))_{C^{*}(\G_{g})}$ is defined from the \'{e}tale Morita equivalence $Z^u(P)$ through the $C^{*}(\G_{g})$ -valued inner product (see \cite[Page 12]{MRW}):
\begin{equation}\label{Moritainnerprod}
\langle f_{1},f_{2}\rangle(v,n,w):=\sum_{(x,j,v)\in Z^{u}(P)} \overline{f_{1}(x,j,v)}f_{2}(x,j+n,w),
\end{equation}
and gives a representation of $C^*(G^s(P))\rtimes \Z$ on the space $L^{2}(Z^{u}(P))_{C^{*}(\G_{g})}\otimes_{C^*(\G_g)} L^2(\G_g)_{C(V)}$. The latter space can be identified with the completion of $C_c(Z^u(P))$ in the $C(V)$-valued inner product
\begin{equation}\label{tensorinnerprod}
\langle f_1,f_2\rangle(v):=\sum_{(x,j,v)\in Z^u(P)} \overline{f_1(x,j,v)}f_2(x,j,v).
\end{equation}
For $f_1,f_2\in C_c(Z^u(P))$ the sums on the right hand side are finite. The $C(V)$-valued inner product \eqref{tensorinnerprod} on $C_c(Z^u(P))$ can be realized by a convolution product $C_c(Z^u(P)^{\op})\times C_c(Z^u(P))\to C_c(\G_g)$ composed with the conditional expectation $\varrho$ on $C_c(\G_g)$. We use the notation $L^2(Z^u(P))_{C(V)}:=L^{2}(Z^{u}(P))_{C^{*}(\G_{g})}\otimes_{C^*(\G_g)} L^2(\G_g)_{C(V)}$.  To avoid notational confusion, we remark that $L^{2}(Z^{u}(P))_{C(V)}$ and $L^{2}(Z^{u}(P))_{C^*(\G_g)}$ are different as vector spaces in general.
\end{remark}

\begin{remark}
\label{uoneremark}
As a consequence of Theorem \ref{moreth}, there are isomorphisms $K_*(C^*(G^s(P)))\cong K_*^{\T}(O_E)$ and $K_*(C^*(G^s(P))\rtimes \Z)\cong K_*(O_E)$. Both these groups can be computed using Corollary \ref{corollarpseqpage13}. It was proven in \cite{KPW} that $C^*(G^s(P))\rtimes \Z$ and $C^*(G^u(P))\rtimes \Z$ are Poincar\'e dual (with dimension shift $1$). In particular, $K_*(C^*(G^u(P))\rtimes \Z)\cong K^{*+1}(C^*(G^s(P))\rtimes \Z)$ can also be computed using Corollary \ref{corollarpseqpage13}, cf. Remark \ref{someremarksonpex}. If the Poincar\'e duality of \cite{KPW} holds $\T$-equivariantly, then $K_*(C^*(G^u(P)))\cong K_{*+1}^{\T}(C^*(G^s(P))\rtimes \Z)$ which could also be computed using Corollary \ref{corollarpseqpage13}.
\end{remark}

\subsection{A closer look at the bimodule $\pmb{C_{c}(Z^{u}(P))}$} 
We wish to compute the Kasparov product of the bimodule induced by $Z^{u}(P)$ (see Theorem \ref{moreth}) with the $\T$-equivariant unbounded $(O_{E}, C(V))$-cycle constructed in Theorem \ref{represent!}. To do so we describe the module $L^{2}(Z^{u}(P))_{C^{*}(\G_{g})}$ obtained by completing the space $C_{c}(Z^{u}(P))$ in the inner product \eqref{Moritainnerprod}. The module has a fairly simple structure. This is due to the fibre product structure of the space $Z^{u}(P)=X^{u}(P)_{\pi_{0}}\times_{r}\G_{g}$ from Proposition \ref{Moritaprop}. We also provide a Cuntz-Pimsner model of $C^*(G^s(P))\rtimes \Z$ along the way.

\begin{lemma}
\label{factorweak} 
Assume that $f\in C_{c}(Z^{u}(P))$ decomposes as a pointwise product $f(z)=f_{1}(z)f_{2}(z)$ with $f_{i}\in C_{c}(Z^{u}(P))$ and $\supp f_{i}\subset W^{s}(z_0,N,W)$, for some $z_0$, $N$ and $W$ (cf. Equation \eqref{basiszup}). Then there exist $u\in C_{c}(X^{u}(P))$ and $v\in C_{c}(\mathcal{R}_{g})$ such that $f(z)=u(\rho_{L}(z))v(\pi_{0}(z))$.
\end{lemma}

\begin{proof} 
For $W$ small enough, $W^{s}(z_0,N,W)$ is homeomorphic to 
$$U:=\rho_{L}(W^{s}(z_0,N,W))=\varphi^{-j}h^{s}_{N}(W),$$ 
and to 
$$V:=\pi_{0}(W^{s}(z_0,N,W))=\mathcal{U}(\pi_{0}(U),N,j,W).$$ 
For a reminder on the notation $\mathcal{U}$, see Equation \ref{basisgg}. Let  $\rho_{L}^{U}:U\to W^{s}(z_0,N,W)$ and $\rho_{L}^{V}:V\to W^{s}(z,N,W)$ denote the inverse mappings. Define $u:=f_{1}\circ\rho_{L}^{U}$ and $v:=f_{2}\circ\rho_{L}^{V}$, and extend to $X^u(P)$ and $\G_g$, respectively, by declaring
\begin{equation}
\label{supps} 
\supp u\subset U, \quad \supp v\subset V.
\end{equation}
In particular, both $u$ and $v$ have compact support. The identity 
\begin{equation}
\label{tranverse}
W^{s}(z_0,N,W)=\rho_{L}^{-1}(U)\cap \pi_{0}^{-1}(V),
\end{equation}
holds because if $(x,j,v)\in \rho_{L}^{-1}(V)$ then $x\in V_{i}=\varphi^{-j}h^{s}_{N}(W)$ and if $(x,j,v)\in \pi_{0}^{-1}(U)$ then $(\pi_{0}(x),j,v)\in U$, so $v\in \pi_{0}(W)$. This means  that $(x,j,v)\in W^{s}(z_0,N,U)$. The other inclusion is obvious. Thus we have that $u(\rho_{L}(z))v(\pi_{0}(z))\neq 0$ implies 
$$z\in \rho_{L}^{-1}(U)\cap \pi_{0}^{-1}(V)=W^{s}(z_0,N,W)$$ 
by \eqref{supps} and \eqref{tranverse}. Then 
$$u(\rho_{L}(z))v(\pi_{0}(z))=f_{1}(\rho_{L}^{U}\rho_{L}(z))f_{2}(\pi^{V}_{0}\pi_{0}(z))=f_{1}(z)f_{2}(z)=f(z),$$ 
as desired.
\end{proof}

\begin{remark} 
Note that because $Z^{u}(P)$ is locally compact and Hausdorff, the decomposition $f=f_{1}f_{2}$ assumed in Lemma \ref{factorweak} can always be achieved if $f$ is supported in a set of the form $W^s(z_0,N,W)$. For our purposes the above formulation of the lemma suffices.
\end{remark}

\begin{lemma}
\label{factorstrong} 
For $f\in C_{c}(Z^{u}(P))$ there exists $n\in\N$ and functions $u_{i}\in C_{c}(X^{u}(P))$ and $f_{i}\in C_{c}(\G_{g})$ such that $f(z)=\sum_{i=1}^{n} u_{i}(\rho_{L}(z))f_{i}(\pi_{0}(z))$.
\end{lemma}

\begin{proof} 
Let $K:=\supp f$ and choose a finite open cover of $K$ sets of the form $W_{i}:=W^{s}(z_{i},N_{i},W_{i})$ for $i=1,\cdots,N$. Adding $W_{0}:=Z^{u}(P)\setminus K$ gives a finite open cover of $Z^{u}(P)$. Let $(\chi_{i}^{2})_{i=1}^n$ be a partition of unity subordinate to $(W_{i})_{i=1}^{n}$, so $$f(z)=\sum_{i=1}^{n}\chi_{i}^{2}(z)f(z)=\sum_{i=1}^{n}\chi_{i}(z)\chi_{i}(z)f(z).$$ The functions $\chi_{i}$ and $\chi_{i}f$ are supported in $W_{i}$. By Lemma \ref{factorweak} there exists functions $u_{i}\in C_{c}(X^{u}(P))$ and $f_{i}\in C_{c}(\G_{g})$ such that $\chi_{i}^{2}(z)f(z)=u_{i}\rho_{L}(z)f_{i}(\pi_{0}(z))$. Thus we have $f(z)=\sum_{i=1}^{n}u_{i}(\rho_{L}(z))f_{i}(\pi_{0}(z))$.
\end{proof}

By Theorem \ref{localfibb}, the map $\pi_{0}:X^{u}(P)\to V$ is a covering map. As such, $\pi_0$ induces a conditional expectation
$$\pi_{0*}:C_{c}(X^{u}(P))\to C(V),\quad \pi_{0*}f(v):=\sum_{x\in \pi_{0}^{-1}(v)} f(x).$$
We denote the $C(V)$-Hilbert $C^{*}$-module completion of $C_{c}(X^{u}(P))$ by $L^{2}(X^{u}(P))_{C(V)}$.
\label{l2xup}

\begin{prop}
\label{tensorfactor} 
Let $(X,\varphi)$ be an irreducible Wieler solenoid arising from an open surjection $g:V \to V$ and suppose $P\subseteq X$ is a finite $\varphi$-invariant set of periodic points. The mapping
\begin{equation}
\label{rholpizer}
\rho_{L}^{*}\otimes \pi_{0}^{*}:C_{c}(X^{u}(P))\otimes_{C(V)} C_{c}(\G_{g})\to C_{c}(Z^{u}(P)),\quad f\otimes h\mapsto \rho_{L}^{*}f\cdot \pi_{0}^{*}h,
\end{equation} 
is a $\T$-equivariant inner product preserving surjection. The mapping \eqref{rholpizer} induces a $\T$-equivariant unitary isomorphism of  $C^{*}$-modules 
$$L^{2}(X^{u}(P))_{C(V)}\otimes_{C(V)}C^{*}(\G_{g})\to L^{2}(Z^{u}(P))_{C^{*}(\G_{g})}.$$ 
Moreover, there exists a $\T$-equivariant isomorphism of $C^*(\G_g)$-Hilbert $C^*$-modules 
\begin{equation}
\label{rholpizertwo}
\mathcal{H}\otimes C^{*}(\G_{g})\cong L^{2}(Z^{u}(P))_{C^{*}(\G_{g})},
\end{equation}
for a separable Hilbert space $\mathcal{H}$ with a trivial $\T$-action.
\end{prop}

\begin{proof} 
To see that the map $\rho_{L}^{*}\otimes \pi_{0}^{*}$ is well-defined, we first assume that $f\otimes h$ is such that $h$ is supported in a basic open neighbourhood $\mathcal{U}(U_{1},k,N, U_{2})$ in $\G_{g}$ which lifts to a basic open neighbourhood $W^{s}(z,N,V_{2})$ and $f$ is supported in $V_{1}:=\rho_{L}(W^{s}(z,N,V_{2}))$. For such $f\otimes h$ we have that 
\[\supp \rho_{L}^{*}(f)\cdot \pi_{0}^{*}(h)\subset \rho_{L}^{-1}(V_{1})\cap \pi_{0}^{-1}(V_{2})=W^{s}(z,N,V_{2}).\]
Therefore, the support of $(\rho_L^*\otimes\pi_0^*)(f\otimes h)$ is compact. Using the balancing relation, a partition of unity argument and Proposition \ref{momentmaps} we have that arbitrary tensors $f\otimes h$ can be written as a sum $f\otimes h=\sum_{i=1}^{n} f_{i}\otimes h_{i}$ for which each $f_{i}$ and $h_{i}$ satisfy the above support requirements. Thus the map $\rho_{L}^{*}\otimes \pi_{0}^{*}$ is well-defined. Surjectivity follows from Lemma \ref{factorstrong}. We prove preservation of the inner products:
\begin{align*}
\langle f_{1}\otimes h_{1}, f_{2}\otimes h_{2}\rangle(v,n,w)&=h^{*}_{1}*\langle f_{1},f_{2}\rangle h_{2}(v,n,w) \\
&=\sum_{(v,j,w)}\sum_{x\in\pi_{0}^{-1}(v)} \overline{h_{1}(u,-j,v)f_{1}(x)}f_{2}(x) h_{2}(u,n-j,w)\\
&=\sum_{(v,j,w)}\sum_{x\in\pi_{0}^{-1}(v)} \overline{\rho_{L}^{*}f_{1}\cdot \pi_{0}^{*}h_{1}(x,-j,v)} \rho_{L}^{*}f_{2}\cdot \pi_{0}^{*}h_{2}(x,n-j,w)\\
&=\sum_{(x,j,v)\in Z^{u}(P)} \overline{\rho_{L}^{*}f_{1}\cdot \pi_{0}^{*}h_{1}(x,j,v)} \rho_{L}^{*}f_{2}\cdot \pi_{0}^{*}h_{2}(x,n+j,w)\\
&=\langle \rho_{L}^{*}f_{1}\cdot \pi_{0}^{*}h_{1}, \rho_{L}^{*}f_{2}\cdot \pi_{0}^{*}h_{2}\rangle \quad\textnormal{by Equation \eqref{Moritainnerprod}}.
\end{align*}
The second statement of the proposition follows immediately. The locally trivial bundle $X^u(P)$ gives rise to a Hilbert space bundle on $V$ which by Kuiper's theorem is trivial. The third statement of the theorem is similar, so that $L^{2}(X^{u}(P))_{C(V)}\cong \mathcal{H}\otimes C(V)$ as $C(V)$-Hilbert $C^*$-modules.
\end{proof}

\begin{cor}
\label{stableisocor}
Let $(X,\varphi)$ be an irreducible Wieler solenoid arising from an open surjection $g:V \to V$ and suppose $P\subseteq X$ is a finite $\varphi$-invariant set of periodic points. The isomorphisms of Theorem \ref{isothemggtooh} and \eqref{rholpizertwo} give rise to a $\T$-equivariant $*$-isomorphism 
$$C^*(G^s(P))\rtimes \Z\cong O_E\otimes \mathbb{K},$$
where $\mathbb{K}$ denotes the $C^*$-algebra of compact operators on a separable Hilbert space with a trivial $\T$-action.
\end{cor}

\begin{remark}
A result similar to Corollary \ref{stableisocor} was proven in \cite[Theorem 4.19]{Tho} for the fixed point algebra of the $\T$-action assuming that $g$ is expansive and the periodic points are dense in $V$. An equivariant version of \cite[Theorem 4.19]{Tho} implies Corollary \ref{stableisocor} under these slightly stronger assumptions on $g$.
\end{remark}

\subsection{A Cuntz-Pimsner model for the stable Ruelle algebra}
It is possible to construct $C^*(G^s(P))\rtimes\Z$ as a Cuntz-Pimnser algebra defined from a non-unital coefficient algebra. We follow the terminology of \cite{AFR}. For notational purposes, we write 
$$\mathbb{K}_V(P):=\mathbb{K}_{C(V)}(L^{2}(X^{u}(P))_{C(V)}).$$
Based on Proposition \ref{tensorfactor}, we define $\tilde{E}:=L^{2}(X^{u}(P))_{C(V)}\otimes E\otimes_{C(V)} L^{2}(X^{u}(P))^*$ which is a $\mathbb{K}_V(P)$-bi-Hilbertian bimodule with finite Jones-Watatani index. The Cuntz-Pimsner algebra $O_{\tilde{E}}$ over the coefficient algebra $\mathbb{K}_V(P)$ is therefore well defined. By Theorem \ref{isothemggtooh}, Theorem \ref{moreth} and Theorem \ref{tensorfactor} there are isomorphisms
\begin{align*}
C^*(G^s(P))\rtimes \Z&\cong \mathbb{K}_{C^*(\G_g)}(L^2(X^u(P))_{C(V)}\otimes C^*(\G_g))\\
&\cong L^2(X^u(P))_{C(V)}\otimes O_E\otimes_{C(V)}L^2(X^u(P))^*.
\end{align*}
Under these isomorphisms, the obvious linear mapping $\tilde{E}\to L^2(X^u(P))_{C(V)}\otimes O_E\otimes_{C(V)}L^2(X^u(P))^*$ induces a linear mapping $\tilde{E}\to C^*(G^s(P))\rtimes \Z$. It is readily verified that this is a covariant representation as in \cite[Theorem 3.12]{Pimsnerspaper} producing a $*$-homomorphism $\tilde{\pi}:O_{\tilde{E}}\to C^*(G^s(P))\rtimes \Z$. The mapping $\tilde{\pi}$ is bijective and we deduce the following result.

\begin{prop}
\label{stablecpa}
Let $(X,\varphi)$ be an irreducible Wieler solenoid arising from an open surjection $g:V \to V$ and suppose $P\subseteq X$ is a finite $\varphi$-invariant set of periodic points. The mapping $\tilde{\pi}:O_{\tilde{E}}\to C^*(G^s(P))\rtimes \Z$ is a $*$-isomorphism. In particular, $\tilde{E}$ provides a Cuntz-Pimsner model for the stable Ruelle algebra with coefficient algebra $\mathbb{K}_V(P):=\mathbb{K}_{C(V)}(L^{2}(X^{u}(P))_{C(V)})$.
\end{prop}

\begin{remark}
Theorem \ref{moreth} and Proposition \ref{stablecpa} are in a sense complementary. The algebra $O_E$ is defined using a unital coefficient algebra, making it easier to work with, but it is only related to $C^*(G^s(P))\rtimes \Z$ via a Morita equivalence. The algebra $O_{\tilde{E}}$ is defined using a non-unital coefficient algebra, but is {\bf explicitly isomorphic} to $C^*(G^s(P))\rtimes \Z$.
\end{remark}

\subsection{The $\kappa$-function on $\pmb{Z^{u}(P)}$} 

To compute the Kasparov product of the cycle in Theorem \ref{represent!} with the Morita equivalence from Theorem \ref{moreth} as an explicit unbounded $(C^{*}(G^{s}(P))\rtimes\Z, C(V))$-cycle we make use of a natural $\kappa$-function defined on $Z^{u}(P)$. To construct a self-adjoint regular operator on the module  $L^{2}(Z^{u}(P))_{C(V)}$ we look at the pullback of the function $\kappa:\G_{g}\to \N$ through the map $\pi_{0}$. Consider the functions
\[\kappa_{Z}(x,j,v):=\min\{k\geq \max\{0,-j\}: g^{k+j}(\pi_{0}(x))=g^{k}(v)\},\quad c_{Z}(x,j,v)=j.\]

\begin{lemma}
\label{thefunctionkappa} 
The functions $\kappa_{Z}$ and $c_{Z}$ are locally constant and hence continuous. 
\end{lemma}

\begin{proof} 
For $c_{Z}$ this is an obvious fact. For $\kappa_{Z}$ this follows from the observation that $\kappa_{Z}=\kappa\circ\pi_{0}$, where $\pi_{0}:Z^{u}(P)\to \G_{g}$.
\end{proof}
 The operator 
\[D_{Z}:=\psi(c,\kappa_{Z}): C_{c}(Z^{u}(P))\to C_{c}(Z^{u}(P)),\] defines a self-adjoint regular operator on $L^{2}(Z^{u}(P))_{C(V)}$. The tensor product module $L^{2}(Z^{u}(P))_{C^{*}(\G_{g})}\otimes_{C^{*}(\G_{g})}L^{2}(\G_{g})_{C(V)}$ can be identified with the module $L^{2}(Z^{u}(P))_{C(V)}$ studied in Proposition \ref{tensorfactor}. On the dense subspace $C_{c}(Z^{u}(P))\otimes_{C_{c}(\G_{g})} C_{c}(\G_{g})$ this identification is realized by the convolution product. We use the standard notation 
\[T_{x}:C_{c}(\G_{g})\to C_{c}(Z^{u}(P))\otimes_{C_{c}(\G_{g})}C_{c}(\G_{g})=C_{c}(Z^{u}(P)),\quad f\to x\otimes f=x*f,\]
defined for elements $x\in C_{c}(Z^{u}(P))$.

\begin{lemma}
\label{connection} 
For $x\in C_{c}(Z^{u}(P))$, the commutator
\[D_{Z}T_{x}-T_{x}D: C_{c}(\G_{g})\to C_{c}(Z^{u}(P)),\]
extends to a bounded operator $L^{2}(\G_{g})_{C(V)}\to L^{2}(Z^{u}(P))_{C(V)}$.
\end{lemma}

\begin{proof} 
By Lemma \ref{factorstrong} and Proposition \ref{tensorfactor}, it suffices to prove the statement for elements $x=\rho_{L}^{*}(u)\cdot \pi_{0}^{*}(v)$, with $u\in C_{c}(X^{u}(P))$ and $v\in C_{c}(\G_{g})$. Moreover, since $$\rho_{L}^{*}(u)\cdot \pi_{0}^{*}(v)=\rho_{L}^{*}(u)\cdot \pi_{0}^{*}(1_{V}*v)=\rho_{L}^{*}(u)\cdot \pi_{0}^{*}(1_{V})*v,$$ and for any $x\in C_{c}(Z^{u}(P))$ 
\[D_{Z}T_{x*v}-T_{x*v}D=(D_{Z}T_{x}-T_{x}D)*v+ T_{x}[D,v],\]
we can further reduce the proof to the case where $x=\rho_{L}^{*}(u)\cdot \pi_{0}^{*}(1_{V})$. For such $x$ and $f\in C_{c}(\G_{g})$ we have
\[(D_{Z}T_{x}-T_{x}D)f(t,j,w)=(\psi(j,\kappa_{Z}(t,j,w))-\psi(j,\kappa(\pi_{0}(t),j,w))u(t)f(\pi_{0}(t),j,w)=0,\]
as desired.
\end{proof}

\begin{remark}
\label{otherapptodz}
The operator $D_Z$ can be constructed from $D$ in other ways. One other way is to choose a cover $(U_j)_{j=1}^N$ of $V$ such that each $\pi_0:\pi_0^{-1}(U_j)\subseteq X^u(P)\to V$ is trivialisable, via a trivialisation $\psi_j$, say. Pick a partition of unity $(\chi_j^{2})_{j=1}^N$ subordinate to $(U_j)_{j=1}^N$. Associated with the data $(U_j,\psi_j,\chi_j)_{j=1}^N$ there is an adjointable inner product preserving $C^*(\G_g)$-linear mapping $v:L^2(Z^u(P))_{C^*(\G_g)}\to \ell^2(\N)\otimes C^*(\G_g)$. A short computation shows that $D_Z=v^*(1\otimes D)v$ on the dense submodule $C_c(Z^u(P))\subseteq L^2(Z^u(P))_{C(V)}$. This is the usual connection construction \cite{MR}, which implies Lemma \ref{connection}.

Yet another way uses Proposition \ref{stablecpa} and the unbounded representative of the Pimsner extension in the non-unital case from \cite{AFR}. The argument in \cite[Subsection 3.2]{GMR} easily generalizes to show that $D_Z$ is the operator constructed in \cite[Theorem 3.7]{AFR}.
\end{remark}

We wish to show that $(L^{2}(Z^{u}(P)_{C(V)},D_{Z})$ is a $KK_1^\T$-cycle for  $(C^*(G^s(P))\rtimes \Z, C(V))$ and that it represents the Kasparov product of the Morita equivalence $L^{2}(Z^{u}(P)_{C^{*}(\G_{g})}$ with the cycle $(L^{2}(\G_{g})_{C(V)},D)$ constructed in Theorem \ref{represent!}. We begin with a useful observation concerning the unbounded Kasparov product with a Morita equivalence bimodule.

\begin{prop}
\label{awesome} 
Let $E$ be a $\T$-equivariant $(A,B)$-Morita equivalence bimodule and $(F,D)$ an odd $\T$-equivariant unbounded Kasparov module for $(B,C)$. Suppose that $\tilde{D}$ is a $\T$-equivariant selfadjoint regular operator on $E\otimes_{B}F$ and that $X\subset E$ is a set that generates $E$ as a $B$-module, such that for all $x\in X$ the operators $T_{x}:f\to x\otimes f$ satisfy 
\begin{equation}
\label{hypcon}
T_{x}:\Dom D\to \Dom\tilde{D},\quad\textnormal{and}\quad\tilde{D}T_{x}-T_{x}D: \Dom D\to E\otimes_{B}F,
\end{equation}
extend to adjointable operators $F\to E\otimes_{B} F$. Then $(E\otimes_{B}F,\tilde{D})$ is a $\T$-equivariant $(A,B)$-unbounded Kasparov module that represents the Kasparov product $[(E,0)]\otimes_{B}[(F,D)]$ in $KK_1^{\T}(A,C)$.
\end{prop}

\begin{proof} 
Since $E$ is a Morita equivalence, $A\cong \mathbb{K}(E)$ and $(E,0)$ is an unbounded $(A,B)$-Kasparov module. The argument in \cite[Lemma 4.3]{MR} and \eqref{hypcon} imply that $(a\otimes 1)(\tilde{D}\pm i)^{-1}$ is compact in $E\otimes_{B}F$ for all $a\in A$. The operators $T_{x}T_{y}^{*}$, with $x,y\in X,$ generate a dense subalgebra of $\mathbb{K}(E)\otimes 1\cong A\otimes 1$. Since
\[ [\tilde{D},T_{x}T_{y}^{*}]=(\tilde{D}T_{x}-T_{x}D)T_{y}^{*}+T_{x}(DT_{y}^{*}-T_{y}^{*}\tilde{D}),\]
is a bounded operator, $(E\otimes_{B} F, \tilde{D})$ is an $(A,C)$ unbounded Kasparov module. To prove that this cycle represents the Kasparov product, we need to verify conditions $(1)-(3)$ as given in \cite[Theorem 13]{Kuc}. Condition $(1)$ is the statement of Equation \eqref{hypcon}. Conditions (2) and (3) are trivially satisfied for the product with the unbounded Kasparov module $(E,0)$ from the left.
\end{proof}

The following result is an immediate consequence of Lemma \ref{connection} and Proposition \ref{awesome}.

\begin{thm}  
\label{productwithzup}
Assume that $(X,\varphi)$ is an irreducible Wieler solenoid defined from an open surjection $g:V\to V$. The pair $(L^{2}(Z^{u}(P))_{C(V)}, D_{Z})$ is a $\T$-equivariant unbounded Kasparov module for the pair $(C^*(G^s(P))\rtimes \Z, C(V))$, that represents the Kasparov product 
$$[(L^{2}(Z^{u}(P))_{C^{*}(\G_{g})},0)]\otimes_{C^{*}(\G_{g})}[(L^{2}(\G_{g})_{C(V)},D)]\in KK_1^{\T}(C^*(G^s(P))\rtimes \Z, C(V)).$$
In particular, there is an exact triangle in the triangulated category $KK^{\T}$ of the form
\[\xymatrix@C=0.8em@R=2.71em{
C(V) \ar[rr]^{1-[E_g]}& & C(V) \ar[dl]^{\iota_\mathcal{R}}\\
  &C^*(G^s(P))\rtimes \Z. \ar[ul]|-\circ^{[L^{2}(Z^{u}(P))_{C(V)}, D_{Z}]\,\,\,\,}
}\]
where $\iota_\mathcal{R}\in KK_0^\T(C(V),C^*(G^s(P))\rtimes \Z)$ is defined from the inclusion $C(V)\hookrightarrow C^*(\G_g)$ and Corollary \ref{stableisocor}.
\end{thm}

\begin{remark}
From the computations in Example \ref{computingktehroyexp}, using Remark \ref{someremarksonpex}, we can compute the $K$-theory group $K_*(C^*(G^s(P)))$ for an expansive dilation matrix on $(S^1)^d$ to be the expression of Equation \eqref{equicompsole} (see page \pageref{equicompsole}).
\end{remark}

\section{KMS weights on the stable Ruelle algebra}
\label{sec: KMS states}

Cuntz-Pimsner algebras come equipped with a natural gauge action that extends to an action of the real numbers. The analysis of equilibrium states (or KMS states) is then of particular interest for these algebras. Such equilibrium states were first studied by Kubo, Martin and Schwinger and a comprehensive study of KMS states is provided by Bratteli and Robinson in \cite{BR}.

Suppose $B$ is a $C^*$-algebra and that $\sigma:\R \to \Aut B$ is a strongly continuous action. An element $a \in B$ is said to be \emph{$\sigma$-analytic} if the function $t \mapsto \sigma_t(a)$ extends to an entire function on $\C$. For $\beta \in (0,\infty]$, a state $\phi$ on $B$ is called \emph{a KMS$_\beta$ state} if it satisfies
\begin{equation}\label{KMS_condition}
\phi(ab)=\phi(b \sigma_{i\beta}(a)),
\end{equation}
for all analytic elements $a,b \in B$. For $\beta=0$, the KMS$_\beta$ states are the $\sigma$-invariant traces on $B$. We often call \eqref{KMS_condition} the KMS$_\beta$ condition.

Olesen and Pedersen \cite{OP} first studied KMS states for the periodic gauge action on the Cuntz algebras $O_{N}$ where they discovered that there is a unique KMS state at inverse temperature $\beta=\log N$. Enomoto, Fujii and Watatani \cite{EFW} generalized this result to the Cuntz-Krieger algebras $O_{\pmb A}$ of an irreducible $N\times N$ matrix ${\pmb A}$. In this case, the unique KMS state occurred at $\beta=\log \rho({\pmb A})$, where $\rho({\pmb A})$ is the spectral radius of the matrix ${\pmb A}$, i.e. its Perron-Frobenius eigenvalue. More recently, Laca and Neshveyev \cite{ln} studied KMS states of Cuntz-Pimsner algebras and proved that KMS states arise as traces on the coefficient algebra of the underlying Hilbert module. In so doing, \cite{ln} initialized a program that is directly related to our work in several situations, most notably for local homeomorphisms  in \cite{aahr} and for self-similar groups in \cite{lrrw}. Similar results for expansive maps are also found in Thomsen's work \cite{Thom KMS} and in Kumjian and Renault's work \cite{kr}.

In this section we first consider the Cuntz-Pimsner algebra $O_E$ associated with a local homeomorphism $g:V \to V$ satisfying Wieler's axioms making the associated Wieler solenoid irreducible, as described in Section \ref{sectioncpmodel}. Recall that the Toeplitz algebra $\TT_E$ is generated by the Toeplitz operators $\{\tilde T_\xi : \xi \in E\}$. Let $S_\xi$ be the image of $\tilde T_\xi$ in the Cuntz-Pimsner algebra $O_E$, which is then generated by $\{S_\xi : \xi \in E\}$. We show that every KMS state for the natural gauge action on $O_E$ gives rise to a KMS weight on the stable Ruelle algebra $C^*(G^s(P)) \rtimes \Z$  through the isomorphism described in Proposition \ref{stablecpa}. 

We let  $\gamma: \T \to \Aut O_E$ denote the strongly continuous gauge action defined in Subsection \ref{TheCuntzPimsneralgebraofalocalhomeomorphism}. The gauge action extends to a periodic action of the real line through the equation $\sigma_t=\gamma_{e^{it}}$. Thus, there is an action $\sigma: \R \to \Aut O_E$ defined on the generating set by $\sigma_t(S_\xi)=e^{it}S_\xi$ for $\xi\in E$.

\begin{thm}
\label{OE KMS states}
Suppose $g:V \to V$ is an open surjection satisfying Wieler's axioms and $O_E$ is the associated Cuntz-Pimsner algebra generated by $\{S_\xi : \xi \in E\}$. Let $\sigma: \R \to \Aut O_E$ denote the strongly continuous action defined by $\sigma_t(S_\xi)=e^{it}S_\xi$ for $\xi\in E$. Then 
\begin{enumerate}
\item\label{KMS existence} 
There is a bijective correspondence between tracial states $\tau$ on $C(V)$ satisfying 
\begin{equation}
\label{KMS formula Laca-Nesh}
\tau(\mathfrak{L}(a))=e^\beta \tau(a)
\end{equation} 
and KMS$_{\beta}$ states on $(O_E,\sigma)$. Through the $\T$-equivariant isomorphism $O_E \cong C^*(\mathcal{R}_g)$, the bijection is given by $\tau \to \omega_\tau$ where  $\omega_\tau$ is defined on $C_c(\G_g)$ by
\begin{equation}
\label{omeagtauadef}
\omega_\tau(f):=\int_V f(v,0,v)\mathrm{d}\mu,
\end{equation}
such that $\mu$ is the probability measure that $\tau$ defines on $V$. There is always at least one $\beta>0$ and one tracial state $\tau$ such that \eqref{KMS formula Laca-Nesh} holds.

\item\label{KMS uniqueness} 
If $(V,g)$ is mixing (Definition \ref{recprops}.\ref{mixing}), then there is a unique pair $(\tau,\beta)$ satisfying \eqref{KMS formula Laca-Nesh}. That is, there is a unique $\beta$ for which there exists a KMS$_{\beta}$ state of $(O_E,\sigma)$ and the KMS$_{\beta}$ state is unique. Moreover, for $v\in V$,
\begin{equation}\label{KMS formula beta}
\beta=h(g)=\lim_{n \to \infty} \log\left(\frac{\mathfrak{L}^{n+1} 1}{\mathfrak{L}^{n} 1}(v)\right),
\end{equation}
where $h(g)$ is the topological entropy of $(V,g)$.
\end{enumerate}
\end{thm}

\begin{remark}
Part \eqref{KMS uniqueness} of Theorem \ref{OE KMS states} contains the statement that $$v\mapsto \lim_{n \to \infty} \log\left(\frac{\mathfrak{L}^{n+1} 1}{\mathfrak{L}^{n} 1}(v)\right)$$ is a constant function on $V$ that equals $\beta$.
\end{remark}

\begin{remark}
In part \eqref{KMS existence} of Theorem \ref{OE KMS states}, we are only using that $g$ is a local homeomorphism. In part \eqref{KMS uniqueness} we are only using that $g$ is a mixing, positively expansive local homeomorphism. We assume Wieler's axioms at this point since the applications we have in mind are to Smale spaces.
\end{remark}

\begin{remark}
The proof of \cite[Lemma 3.4]{GMR} extends to local homeomorphisms. This fact shows that the conditional expectation $\Psi_\infty:O_{E_g}\to C(V)$, constructed for a large class of Cuntz-Pimsner algebras in \cite{RRS}, satisfies $\Psi_\infty(f)(x):=f(x,0,x)$ for $f\in C_c(\G_g)$. The correspondence in Theorem \ref{OE KMS states} part \eqref{KMS existence} between tracial states $\tau$ satisfying $\mathfrak{L}^*\tau=\mathrm{e}^\beta\tau$ and KMS$_\beta$ states on $O_{E_g}$ is then implemented by $\tau\mapsto \tau\circ \Psi_\infty$. In \cite[Section 4]{RRS}, a related correspondence was considered for a larger class of modules using an $\R$-action defined from the Jones-Watatani index. In the case under consideration in this paper, the Jones-Watatani indices are 1 (cf. \cite[Proposition 3.3]{GMR}) and the $\R$-action considered in  \cite[Section 4]{RRS} is trivial.
\end{remark}

\begin{proof}[Proof of Theorem \ref{OE KMS states}]
We will apply \cite[Theorems 2.1 and 2.5]{ln} and start by verifying the hypothesis of said theorems. Since the gauge dynamics on $E$ corresponds to the action of scalars on $C(V)$, the `positive energy' hypothesis of \cite[Theorem 2.1]{ln} holds; the Arveson spectrum $\text{Sp}_U(a)$ is equal to $1$ for all $a \in E$. Thus, in our situation, combining \cite[Theorems 2.1 and 2.5]{ln} implies that $\psi$ is a KMS$_\beta$ state on $O_E$ if and only if there is a tracial state $\tau$ on $C(V)$ such that $\tau(\mathfrak{L}(a))=e^\beta \tau(a)$ and for $\xi\in E^{\otimes m}$ and $\eta\in E^{\otimes n}$
$$
\psi(S_\xi S_\eta^*)=
\begin{cases}
\mathrm{e}^{-\beta m}\tau(\langle \eta,\xi\rangle_{E^{\otimes m}}), \;&\mbox{if  $n=m$}\\
0, \;&\mbox{if  $n\neq m$} \end{cases}.$$
From this identity and the fact $\tau(\mathfrak{L}(a))=e^\beta \tau(a)$, we deduce that 
\begin{equation}
\label{psicomp}
\psi(S_\xi S_\eta^*)=\begin{cases}
\mathrm{e}^{-\beta m}\tau(\mathfrak{L}^m(\overline{\eta}\xi)), \;&\mbox{if  $n=m$}\\
0, \;&\mbox{if  $n\neq m$} \end{cases}=\begin{cases}
\tau(\overline{\eta}\xi), \;&\mbox{if  $n=m$}\\
0, \;&\mbox{if  $n\neq m$} \end{cases},
\end{equation}
where we use the identification $E^{\otimes m}\cong C(V)$ as linear spaces. Since $V$ is a compact Hausdorff space, every tracial state $\tau$ on $C(V)$ is given by integrating against a probability measure. From Equation \eqref{psicomp} we deduce that $\psi$ is a KMS$_\beta$ state on $O_E$ if and only if there is a probability measure $\mu$ on $V$ such that 
\begin{equation}
\label{KMS Walters cond}
\mathfrak{L}_* \mu=e^\beta \mu,
\end{equation}
and $\psi=\omega_\tau$, where $\omega_\tau$ is defined in Equation \eqref{omeagtauadef}.

The existence of probability measures $\mu$ on $V$ satisfying \eqref{KMS Walters cond} was taken up by Walters in \cite{Wal}. If $g:V \to V$ is expansive, \cite[Corollary 2.3]{Wal} implies that there is at least one pair $(\mu,\beta)$ satisfying \eqref{KMS Walters cond}. Since $g$ is assumed to be a local homeomorphism, it follows that $(V,g)$ is expansive (see Remark \ref{remarkaboutthomsen}), and therefore \eqref{KMS existence} holds.

In the case that $g: V \to V$ is a positively expansive local homeomorphism satisfying the weak specification condition, \cite[Theorem 2.16(i)]{Wal} implies that there is a unique pair $(\mu,\beta)$ on $(V,g)$ such that \eqref{KMS Walters cond} holds. In \cite[Proposition 2.1]{kr}, Kumjian and Renault show that the weak specification property is equivalent to the condition that $(V,g)$ is mixing. Moreover, \cite[Theorem 2.16(iv)]{Wal} shows that the unique value $\beta$ satisfies \eqref{KMS formula beta}. Therefore \eqref{KMS uniqueness} holds as well.
\end{proof}

Before turning to the $C^*$-algebra $C^*(G^s(P))\rtimes \Z$, we note the following result which follows from a short algebraic manipulation with matrix units.

\begin{prop}
\label{proponkms}
Let $\mathbb{K}$ denote the $C^*$-algebra of compact operators on a separable Hilbert space with trivial $\R$-action and $\mathrm{tr}$ the operator trace on $\mathbb{K}$. If $B$ is a unital $\R-C^*$-algebra, then the mapping $\omega\mapsto \omega\otimes \mathrm{tr}$ defines a bijection between the KMS$_\beta$ states on $B$ and the KMS$_\beta$ weights on $B\otimes \mathbb{K}$. 
\end{prop}

\begin{cor}
\label{kmsweightsonds}
Let $(X,\varphi)$ be an irreducible Wieler solenoid arising from an open surjection $g:V \to V$ and suppose $P\subseteq X$ is a finite $\varphi$-invariant set of periodic points. There is a bijective correspondence between measures $\mu$ on $V$ satisfying \eqref{KMS Walters cond} and KMS$_\beta$ weights on $C^*(G^s(P)) \rtimes \Z$ via $\mu\mapsto \phi_\mu$ where the KMS$_\beta$ weight $\phi_\mu$ is defined on $C_c(\mathcal{G}^s(P)\rtimes \Z)$ via 
$$\phi_\mu(f):=\int_{X^u(P)} f(x,0,x)\mathrm{d}(\pi_0^*\mu).$$
In particular, $C^*(G^s(P))\rtimes \Z$ always admits a KMS$_\beta$ weight and if $(V,g)$ is mixing, then there is a unique $\beta>0$ for which there exists a KMS$_\beta$ weight and that KMS$_\beta$ weight is unique.
\end{cor}

\begin{proof}
Since $\pi_0:X^u(P)\to V$ is a covering (see Theorem \ref{localfibb}), the measure $\pi_0^*\mu$ is well-defined on $X^u(P)$. We choose a $\T$-equivariant $*$-isomorphism 
\[
\rho:C^*(G^s(P))\rtimes \Z\xrightarrow{\cong} O_E\otimes \mathbb{K}
\]
as in Corollary \ref{stableisocor}. By Proposition \ref{proponkms}, there is a bijective correspondence between KMS$_\beta$ states $\omega$ on $O_E$ and KMS$_\beta$ weights on $C^*(G^s(P))\rtimes \Z$ given by $\omega\mapsto (\omega\otimes \mathrm{tr})\circ \rho$. By Theorem \ref{OE KMS states} implies that any KMS$_\beta$ state on $O_E$ takes the form $\omega_\tau$ where $\tau\in C(V)^*$ is defined from a probability measure $\mu$ on $V$ satisfying Equation \eqref{KMS Walters cond}. Since $\mathrm{tr}$ is invariant under unitary transformations and $\rho$ is $C(V)$-bilinear, it follows from an argument in local trivializations of $\pi_0:X^u(P)\to V$ that for $f\in C_c(G^s(P)\rtimes \Z)$ we have 
$$(\omega_\tau\otimes \mathrm{tr})\circ \rho(f)=\int_V\Big[\sum_{x\in \pi_0^{-1}(v)}f(x,0,x)\Big]\mathrm{d} \mu(v)=\int_{X^u(P)} f(x,0,x)\mathrm{d}(\pi_0^*\mu)(x).$$

From Theorem \ref{OE KMS states} we deduce that $C^*(G^s(P))\rtimes \Z$ always admits a KMS$_\beta$ weight and if $(V,g)$ is mixing the KMS$_\beta$ weight is unique. The corollary now follows from \cite[Theorem 3.5.3]{AH}, which says that if $g:V \to V$ is a continuous surjection, then $g:V \to V$ is mixing if and only if $\varphi:X \to X$ is mixing.
\end{proof}

\begin{remark}
\label{remarkonforms}
If $V$ is an oriented manifold and $\tau(a)=\int_V a\wedge \omega$ for a differential form $\omega$, the condition $\mathfrak{L}^*\tau=\mathrm{e}^\beta\tau$ is equivalent to $g^*\omega=\mathrm{e}^{\beta}\omega$.
\end{remark}

\section{Examples}
\label{examplessection}

In this section, we will consider examples of irreducible Wieler solenoids arising from an open surjection $g:V\to V$. In these examples, Theorem \ref{productwithzup} enables computations of explicit $K$-theoretic invariants and construction of concrete spectral triples.

\subsection{Subshifts of finite type}
\label{cksubsec}

The Smale spaces where both the stable and unstable sets are totally disconnected are the subshifts of finite type. They were discussed in Example \ref{firstckex}. In this case, the stable and unstable Ruelle algebra can be treated on an equal footing. Consider an $N\times N$-matrix $\pmb{A}$ consisting of zeros and ones, and assume that $\pmb{A}$ is irreducible. The compact space $\Omega_{\pmb{A}}$ is defined in Example \ref{firstckex} and the associated Wieler solenoid  $(X_{\Omega_{\pmb{A}}},\varphi_{\pmb A})$  is conjugate to the two sided sided shift
$$\Sigma_{\pmb{A}}:=\{(x_j)_{j\in \Z}\in \{1,\ldots,N\}^\Z: \pmb{A}_{x_j,x_{j+1}}=1\}, \quad \varphi_{\pmb{A}}(x_j)_{j\in \Z}:=(x_{j+1})_{j\in \Z}.$$
As above, we chose a finite $\varphi_{\pmb{A}}$-invariant set $P\subseteq \Sigma_{\pmb{A}}$ of periodic points. We write 
$$Z^u_{\pmb{A}}(P):=\{(x,j,v)\in \Sigma_{\pmb{A}}^u(P)\times \Z\times \Omega_{\pmb{A}}: \; \exists k, \, x_{l+j}=v_l\;\forall l\geq k\}.$$
It follows from Theorem \ref{moreth} that $C^*(G^s_{\pmb{A}}(P))\rtimes \Z\sim_M O_{\pmb{A}}$ via the $\T$-equivariant Morita equivalence $Z^u_{\pmb{A}}(P)$. The $K$-theory and $K$-homology groups of the Cuntz-Krieger algebra $O_{\pmb{A}}$ are well-known, and given by
$$K^*(O_{\pmb{A}})\cong K_{*+1}(O_{\pmb{A}^T})\cong 
\begin{cases}
\mathrm{ker}(1-\pmb{A}), \quad& *=0\\
\mathrm{coker}(1-\pmb{A}), \quad& *=1.\end{cases}$$
Since $\Omega_{\pmb{A}}$ is totally disconnected, $K^1(C(\Omega_{\pmb{A}}))=0$ and the exact $K$-homology sequence in Corollary \ref{corollarpseqpage13} reduces to
\[0\to K^{0}( C^{*}(G^{s}(P)\rtimes\Z)) \xrightarrow{\iota^{*}} K^0(C(\Omega_{\pmb{A}}))\xrightarrow{1-[E_g]} K^0(C(\Omega_{\pmb{A}})) \xrightarrow{\partial} K^1(C^{*}(G^{s}(P)\rtimes \Z))\to 0,\]
In particular the boundary map $\partial$ is surjective, and the description of explicit representatives of $K^1(O_{\pmb{A}})\cong \Z^N/(1-\pmb{A})\Z^N$ in \cite{goffmes} shows that $\partial$ is surjective already when restricted to classes in $K^{0}(C(\Omega_{\pmb{A}}))$ given by point evaluations of $C(\Omega_{\pmb{A}})$: each class in $K^{1}(O_{\pmb{A}})$ is given by a spectral triple obtained by localizing the unbounded Kasparov module from Theorem \ref{represent!} in a character of $C(\Omega_{\pmb{A}})$. We see that the same statement is true for the stable Ruelle algebra of the two-sided shift. Let $(e_i)_{i=1}^N$ denote the standard basis of $\Z^N$. We deduce the following result from \cite[Theorem 5.2.3]{goffmes}.

\begin{prop}
\label{proponcktwosided}
Let $v\in \Omega_{\pmb{A}}$ be a one-sided sequence starting in the letter $i\in \{1,\ldots, N\}$. Define the Hilbert space 
$$\mathcal{H}_v:=\ell^2(\rho_R^{-1}(v)),\quad\mbox{where $\rho_R^{-1}(v)=\{(x,n)\in \Sigma_{\pmb{A}}^u(P)\times \Z: \, (x,n,v)\in Z^u_{\pmb{A}}(P)\}$.}$$
We let $C^*(G^s(P))\rtimes \Z$ act on $\mathcal{H}_v$ via identifying $\mathcal{H}_v$ with the localisation of $L^2(Z^u(P))_{C(V)}$ in $v\in \Omega_{\pmb{A}}$. Consider the self-adjoint operator $D_v$ defined by 
$$D_vf(x,n):=\psi(n,\kappa_Z(x,n,v))f(x,n).$$ 
Then $(C_c(G^s(P))\rtimes_{\rm alg} \Z, \mathcal{H}_v,D_v)$ is a $\theta$-summable spectral triple representing the equivalence class $e_i\!\!\mod (1-\pmb{A})\Z^N$ under the isomorphism $K^1(O_{\pmb{A}})\cong \Z^N/(1-\pmb{A})\Z^N$. The bounded Fredholm module $(C_c(G^s(P))\rtimes_{\rm alg} \Z, \mathcal{H}_v,D_v|D_v|^{-1})$ is finitely summable.
\end{prop}
This provides an exhaustive explicit description of the odd $K$-homology of the stable Ruelle algebra of a subshift of finite type. The group $K^0(C^*(G^s(P)\rtimes\Z))$ stands in stark contrast; a description of even spectral triples on the stable Ruelle algebra is at this point in time still elusive.

\subsection{Self-similar groups}
\label{selfsimilarsubsec}

Self-similar groups are a relatively new area of group theory. They were discovered by Grigorchuk through his construction of a group with intermediate growth that is amenable but not elementary amenable. Since then much of the abstract theory has been developed by Nekrashevych \cite{Nek_book}.

In \cite{Nek_crelle}, Nekrashevych constructs a Smale space from a contracting and regular self-similar group. In this section we recast Nekrashevych's result by showing that the limit space of a contracting, regular self-similar group satisfies Wieler's axioms and the shift map is a local homeomorphism. The shift on the limit space of a self-similar group gives rise to another Cuntz-Pimsner model for self-similar groups that generalizes the Cuntz-Pimsner algebra $O_{f}$ constructed by Nekrashevych in \cite[Section 6.3]{Nek_crelle}. We note that the results of Section \ref{sectionfour}, \ref{sectionfive} and \ref{sec: KMS states} apply to this example. We obtain a plethora of Kasparov cycles, KMS states on the Cuntz-Pimsner algebra $O_E$ and KMS weights on the stable Ruelle algebra. We also remark that all of these results generalize to Exel and Pardo's self-similar groups on graphs through the forthcoming paper \cite{BBGSW}.

A self-similar group $(G,X)$ consists of a group $G$, a finite set $X$, and a faithful action of $G$ on the set $X^*$ of finite words in $X$ such that, for each $g \in G$ and $x \in X$, there exists $y \in X$ and $h \in G$ satisfying $g \cdot xw=y(h\cdot w)$ for all $w \in X^*$. Faithfulness of the action implies the group element $h$ is unique, so there is a map $(g,x) \mapsto h$ and we call $h$ the \emph{restriction} of $g$ to $x$ and denote it by $g|_x$. Thus the defining relation for a self-similar group can be written
\begin{equation}\label{SSG def}
g \cdot xv=(g \cdot x)(g|_x \cdot w) \qquad \text{ for all $w \in X^*$.}
\end{equation}
A self-similar group is said to be \emph{contracting} if there is a finite subset $S \subset G$ such that for every $g \in G$, there exists $n \in \N$ such that $\{g|_v : |v| \geq n\} \subset S$. The (unique) smallest such set is called the \emph{nucleus} as is denoted by $\NN$.

We now briefly construct the limit space $\JJ_G$ in the case that $(G,X)$ is contracting, for further details see \cite[Section 3.6]{Nek_book}. The space of left infinite words with letters in $X$ is denoted $X^{-\infty}$. Two sequences $\dots x_3x_2x_1$ and $\dots y_3y_2y_1$ are said to be \emph{asymptotically equivalent} if there exists a sequence $\{g_n\}_{n=1}^\infty$ of nucleus elements such that $g_n \cdot x_n \dots x_1=y_n \dots x_1$. The \emph{limit space} $\JJ_G$ is the quotient of $X^{-\infty}$ by the asymptotic equivalence relation, and we let $q:X^{-\infty} \to \JJ_G$ denote the quotient map. The asymptotic equivalence relation is invariant under the shift map $\sigma:X^{-\infty} \to X^{-\infty}$ given by $\sigma(\dots x_3x_2x_1)=\dots x_3x_2$, so $\sigma:\JJ_G \to \JJ_G$ is a continuous surjection. That $\JJ_G$ is a compact metric space is proved in \cite[Theorem 3.6.3]{Nek_book}, and in a more general framework in \cite{BBGSW}.

In \cite{BBGSW}, it is proved that $(\JJ_G,\sigma)$ satisfies Wieler's axioms in the generality of Exel and Pardo's self-similar graphs. The proofs require an additional assumption: a self-similar group is said to be \emph{regular} if for every $g \in G$ and $x \in X^{\infty}$, either $g \cdot x \neq x$ or $g$ fixes $x$ as well as all points in a clopen neighbourhood of $x$. In \cite[Proposition 6.1]{Nek_crelle}, Nekrashevych shows that $(G,X)$ is regular if and only if $\sigma:\JJ_G \to \JJ_G$ is a covering. We now summarize the results of \cite{BBGSW}.

\begin{prop}[{\cite{BBGSW}}]\label{SSGs locally expanding}
Suppose $(G,X)$ is a contracting and regular self-similar group with limit space $\JJ_G$. Then $\sigma: \JJ_G \to \JJ_G$ is a local homeomorphism that satisfies Wieler's axioms and the dynamical system $(\JJ_G,\sigma)$ is mixing.
\end{prop}

Let $(X,\varphi)$ be the mixing Wieler solenoid associated with $(\JJ_G,\sigma)$. Since $\sigma:\JJ_G\to \JJ_G$ is a continuous open surjection satisfying Wieler's axioms, there is a Cuntz-Pimsner model for the stable Ruelle algebra of $(X,\varphi)$ by Proposition \ref{stablecpa}. Moreover, Corollary \ref{kmsweightsonds} implies that there is a unique KMS weight on the stable Ruelle algebra of a contracting and regular self-similar group.

\section{Solenoids constructed from covering maps on manifolds}
\label{solenoidsubsec}

In this section we will consider the solenoid construction on a smooth closed manifold $M$. The outcome will be a procedure to lift certain $K$-homology classes on $M$ to the Cuntz-Pimsner algebra constructed from $M$. The results are complementary to the constructions in Section \ref{sectionfour} and \ref{sectionfive} as it produces $K$-homology classes that are not factorizable over the Pimsner extension in the sense of Corollary \ref{corollarofpseqonpage14}. In certain special cases, the Hilbert space appearing in the spectral triple is the GNS space of the KMS weight. To elucidate our results we will run the details of $M=S^1$ in Example \ref{solenioidexampleonaroundpage7} in parallel to the general construction. We show that for $M=S^1$, the $K$-homology classes we construct along with the products with the Pimsner extension exhaust the $K$-homology of the stable Ruelle algebra $C^*(G^s(P))\rtimes \Z$.

We consider a closed manifold $M$ and an $n$-fold smooth covering map $g:M\to M$ acting conformally on a Riemannian metric $h_M$ on $M$. That is, $g^*h_M=Nh_M$ for $N\in C^\infty(M,\R_{>0})$. The function $g$ satisfies Wieler's axioms if $N>1$ on $M$. We follow the notation from the previous sections and write $\mathfrak{L}(\xi)(x):=\sum_{g(y)=x} \xi(y)$ for $\xi\in C(M)$.

\begin{remark}
It follows from Remark \ref{remarkonforms} that the volume measure $\mathrm{d}V_h$ satisfies Equation \eqref{KMS Walters cond} with $\beta=\frac{d}{2}\log(N)$ if $N$ is constant. Specifically, if $N$ is constant, then the volume measure defines a KMS$_\beta$ state on $O_{E_g}$ (as in Theorem \ref{OE KMS states}) and a KMS$_\beta$ weight on $C^*(G^s(P))\rtimes \Z$ (as in Corollary \ref{kmsweightsonds}).
\end{remark}

\begin{definition}
Let $S\to M$ denote a hermitian vector bundle with hermitian metric $h_S$ such that $g$ lifts to $S$ (in the sense that there is a unitary vector bundle isomorphism $g_S:g^*S\to S$ giving rise to a unitary isomorphism $g_{S,y}:S_{g(y)}\to S_y$ for any $y\in M$). We  define the transfer operator
$$\mathfrak{L}_S:C(M,S)\to C(M,S), \quad\mathfrak{L}_S(\xi)_x:=\sum_{g(y)=x} g_{S,y}^{-1}\xi_y.$$
\end{definition}

The operator $\mathfrak{L}_S$ is $C(M)$-linear in the sense that $\mathfrak{L}_S(\xi\cdot g^*a)=\mathfrak{L}_S(\xi)a$ for any $a\in C(M)$ and $\xi\in C(M,S)$. Define the mapping $V_0:L^2(M,S)\to L^2(M,S)$ as the composition of the pullback $g^*:L^2(M,S)\to L^2(M,g^*S)$ and the fibrewise action of $g_S$.

\begin{prop}
\label{theisometry}
Assume that $M$ is a closed $d$-dimensional manifold with the data specified above. The operator $V:=N^{d/4}n^{-1/2}V_0$ is an isometry on $L^2(M,S)$ such that 
$$V^*=\mathfrak{L}_SN^{-d/4}n^{-1/2},\quad VaV^*=g^*a VV^*\quad\mbox{and}\quad V^*aV=\frac{1}{n}\mathfrak{L}(a).$$
\end{prop}

\begin{proof}
We first prove that $V$ is an isometry. Let $\mathrm{d}V_h$ denote the volume density constructed from $h_M$. Note that $g^*\mathrm{d}V_h=N^{d/2}\mathrm{d}V_h$. We have for $f_1,f_2\in C(M,S)$ that
\begin{align*}
\langle Vf_1,Vf_2\rangle_{L^2(M,S)}&=\frac{1}{n}\int_M\langle g^*f_1,g^*f_2\rangle_SN^{d/2}\mathrm{d}V_h\\
&=\frac{1}{n}\int_Mg^*\left(\langle f_1,f_2\rangle_S\mathrm{d}V_h\right)=\langle f_1,f_2\rangle_{L^2(M,S)}.
\end{align*}
Next, we prove the identity $V^*aV=\frac{1}{n}\mathfrak{L}(a)$. For $f_1,f_2\in C(M,S)$, a similar computation as above gives
\begin{align*}
\langle V^*aVf_1,f_2\rangle_{L^2(M,S)}&=\langle aVf_1,Vf_2\rangle_{L^2(M,S)}=\frac{1}{n}\int_Mag^*\left(\langle f_1,f_2\rangle_S\mathrm{d}V_h\right)\\
&=\frac{1}{n}\int_M\mathfrak{L}(a)\langle f_1,f_2\rangle_S\mathrm{d}V_h=\left\langle \frac{1}{n}\mathfrak{L}(a)f_1,f_2\right\rangle_{L^2(M,S)}.
\end{align*}
The identity $VaV^*=g^*a VV^*$ follows from the simple computation $V(af)=g^*a\cdot Vf$ which holds for $a\in C(M)$, $f\in C(M,S)$.

Finally, we compute $V^*$. For $f_1,f_2\in C(M,S)$, 
\begin{align*}
\langle Vf_1,f_2\rangle_{L^2(M,S)}&=\int_M\langle g^*f_1,N^{-d/4}n^{-1/2}f_2\rangle_S N^{d/2}\mathrm{d}V_h\\
&=\int_M\left\langle f_1,\mathfrak{L}_S\left(N^{-d/4}n^{-1/2}f_2\right)\right\rangle_S \mathrm{d}V_h. \qedhere
\end{align*}
\end{proof}

\begin{example}
\label{solenoidsandisometries}
The prototypical example is the $n$-solenoid on $M=S^1$ from Example \ref{solenioidexampleonaroundpage7}. We return to this example throughout the section. In this example, $M=S^1=\R/\Z$ and $g(x):=nx\pmod 1$. We can take $S$ to be the trivial line bundle on $S^1$ and $h_{S^1}$ to be the flat metric on $S^1$ so $N=n^2$ and $Vf(x)=f(nx)$. An orthonormal basis for $L^2(S^1)$ is given by $e_k(x):=\mathrm{exp}(2\pi i kx)$ for $k\in \Z$. In this case, the identity $Ve_k=e_{kn}$ show straight away that $V$ is an isometry. In this case, the transfer operator $\mathfrak{L}$ takes the form
$$\mathfrak{L}(a)(x)=\sum_{j=0}^{n-1}a\left(\frac{x+j}{n}\right).$$
Here we identify functions on $S^1$ with periodic functions on $\R$. On $L^2(S^1)$ we have $V^*=n^{-1}\mathfrak{L}$.
\end{example}

\begin{example}
\label{dilationnumber}
We consider a higher dimensional example mentioned towards the end of Example \ref{solenioidexampleonaroundpage7}. Consider a matrix $A\in M_d(\Z)$ with non-zero determinant and let $g_A:(S^1)^d\to (S^1)^d$ denote the associated smooth local homeomorphism. If $A^TA=N$ for some $N>0$, then $g_A$ act conformally on the flat metric on $(S^1)^d=\R^d/\Z^d$. Let us give some non-trivial examples of such matrices. For $d=2$, such a matrix takes the general form
$$\begin{pmatrix} x&\pm y\\ y& \mp x\end{pmatrix},\quad x^2+y^2=N.$$
Asymptotically, the number of local homeomorphisms acting conformally on $(S^1)^2$ with $N\leq r^2$ is determined by Gauss' circle problem and behave like $\pi r^2$ as $r\to \infty$. The equation $x^2+y^2=N$ admits solutions if and only if the prime factors $p|N$ with $p=3\pmod 4$ occur with even multiplicity in $N$. In dimension $d=3$, the general form is
$$\begin{pmatrix} x_1&y_1&z_1\\ x_2&y_2&z_2\\ x_3&y_3&z_3\end{pmatrix},\quad\mbox{where} \quad
\begin{cases} 
x_j^2+y_j^2+z_j^2=N, \mbox{   for $j=1,2,3$}\\
x_jx_k+y_jy_k+z_jz_k=0, \mbox{   for $j\neq k$.}\end{cases} $$
By Legendre's three square theorem, such a matrix exists only if $N$ is not of the form $4^a(8b+7)$ for natural numbers $a,b$.

In this case, the Lebesgue measure $m$ on $(S^1)^d$ satisfies $g_A^*m=N^{d/2}m$. Theorem \ref{OE KMS states} shows that the Lebesgue measure induces the unique KMS state on $O_{E_{g_A}}$ and by Corollary \ref{kmsweightsonds} the unique KMS weight on $C^*(G^s(P))\rtimes \Z$, both having inverse temperature $\beta=\frac{d}{2}\log(N)$.
\end{example}

\begin{prop}
\label{constrrep}
Let $M$ be a $d$-dimensional manifold, $g:M\to M$ an $n$-fold smooth covering map acting conformally on a metric $h_M$ that lifts to the hermitian bundle $S\to M$, and $V$ the isometry on $L^2(M,S)$ constructed as in Proposition \ref{theisometry}. Write $E_g:={}_{\mathrm{id}}C(M)_{g^*}$ as in Subsection \ref{TheCuntzPimsneralgebraofalocalhomeomorphism}. The pointwise action of $C(M)$ on $L^2(M,S)$ and the linear mapping $\mathfrak{t}_V:E_g\to \mathbb{B}(L^2(M,S))$ given by $a\mapsto \sqrt{n}aV$, defines a $*$-representation $\pi_V:O_{E_g}\to \mathbb{B}(L^2(M,S))$.
\end{prop}

\begin{proof}
It is an immediate consequence of Proposition \ref{theisometry} that $\mathfrak{t}_V$ defines a Toeplitz representation of $E_g$; that is, conditions (1)--(3) of \cite[Theorem 3.12]{Pimsnerspaper} are satisfied. It remains to prove that the Cuntz-Pimsner covariance condition (4) of \cite[Theorem 3.12]{Pimsnerspaper} is satisfied. Thus, we need to show the following: after writing the action of $a\in C(M)$ on $E_g$ as $a\xi=\sum_{j=1}^m\eta_j\langle \xi_j,\xi\rangle_{C(M)}$, for some $\eta_j,\xi_j\in E_g$, we require that $a=\sum_{j=1}^m\mathfrak{t}_V(\eta_j)\mathfrak{t}_V(\xi_j)^*$ as operators on $L^2(M,S)$.

As in the discussion preceding \eqref{omegadefin}, consider a cover $(U_j)_{j=1}^m\subseteq M$ of balls such that $g|_{U_j}$ is injective. Take a partition of unity $(\chi_j^{2})_{j=1}^m$ subordinate to $(U_j)_{j=1}^m$, so that $\chi_j$ is a frame for $E_g$; that is, for any $\xi\in E_g=C(M)$ we have $\xi=\sum_{j=1}^m \chi_j\langle \chi_j,\xi\rangle_{C(M)}$. So we can write $a\xi=\sum_{j=1}^m\eta_j\langle \xi_j,\xi\rangle_{C(M)}$ where $\eta_j=a\chi_j$ and $\xi_j=\chi_j$. We fix this decomposition for the remainder of the proof.

We now verify the Cuntz-Pimsner covariance condition. Consider an $f\in C(M,S)$. Using Proposition \ref{theisometry}, we write
\small
\begin{align*}
\left[\sum_{j=1}^m\mathfrak{t}_V(\eta_j)\mathfrak{t}_V(\xi_j)^*f\right](x)&=n\left[\sum_{j=1}^m\eta_jVV^*\xi_jf\right](x)=\left[\sum_{j=1}^m\eta_jN^{d/4}V_0\mathfrak{L}_SN^{d/4}\xi_jf\right](x)\\
&=\sum_{j=1}^m\sum_{g(x)=g(y)}a(x)\chi_j(x)N^{d/4}(x)N^{d/4}(y)\chi_j(y)g_{S,x}g_{S,y}^{-1}f(y)\\
&=\sum_{j=1}^ma(x)\chi^{2}_j(x)f(x)=a(x)f(x),
\end{align*}
\normalsize
where the second last identity used the fact that $g|_{U_j}$ is injective, so that $$\{y\in \supp(\chi_j):g(x)=g(y)\}=\{x\},$$ for $x\in \supp (\chi_j)$.
\end{proof}

\begin{definition}
Define the function $\mathrm{signlog}:\R\to \R$ by 
$$\mathrm{signlog}(x):=\begin{cases}
0, \quad &x=0,\\
\mathrm{sign}(x)\log|x|, \quad &x\neq 0.
\end{cases}$$
\end{definition}

We set $\Dsla_{\mathrm{log}}:=\mathrm{signlog}(\Dsla)$. For any $\ep>0$, there is a function $\mathfrak{l}_\ep\in C^\infty(\R)$ such that $\mathfrak{l}_\ep=\mathrm{signlog}$ on $\R\setminus \big((-\ep,0)\cup(0,\ep)\big)$. We remark that for any $\ep>0$ and $m>0$, the function $\mathfrak{l}_\ep$ is a H\"ormander symbol of order $m$ on $\R$. In fact, for any $k\in \N_{>0}$ there is a $C_k>0$ (possibly depending on $\ep>0$) such that $|\partial^k_x\mathfrak{l}_\ep(x)|\leq C_k(1+|x|)^{-k}$. We note that $\Dsla_{\mathrm{log}}-\mathfrak{l}_\ep(\Dsla)\in \Psi^{-\infty}(M,S)$ for all $\ep>0$. For $\ep$ small enough, $\Dsla_{\mathrm{log}}=\mathfrak{l}_\ep(\Dsla)$ if $\mathfrak{l}_\ep(0)=0$. 

\begin{prop}
\label{signlogopstruc}
Let $M$ be a closed Riemannian manifold, $S\to M$ a Clifford bundle and $\Dsla$ a Dirac operator acting on $S$. Then $\Dsla_{\mathrm{log}}$ is a self-adjoint operator and the inclusion $H^s(M,S)\subseteq \Dom(\Dsla_{\mathrm{log}})$ is compact for any $s>0$, and $e^{-t\Dsla_{\mathrm{log}}^2}$ is in the operator ideal $\mathcal{L}^{(d/t,\infty)}(L^2(M,S))$ for any $t>0$.
\end{prop}

Proposition \ref{signlogopstruc} follows from functional calculus for self-adjoint operators. To study $\Dsla_{\mathrm{log}}$ further, we will make use of log-polyhomogeneous pseudo-differential operators. Such operators are studied in detail in \cite[Section 3]{logpolylesch}. In short, the log-polyhomogeneous pseudo-differential operators form a bi-filtered algebra $\Psi^{m,k}(M,S)\subseteq \cap_{s>0}\Psi^{m+s}(M,S)$ where $m\in \R$ and $k\in \N$. A full symbol $\sigma_U^{\rm full}(A)$ of an operator $A\in \Psi^{m,k}(M,S)$ in a local coordinate chart $U$ admits an asymptotic expansion 
$$\sigma_U^{\rm full}(A)(x,\xi)\sim \sum_{j=0}^k a_j(x,\xi)\log^j|\xi|,$$
where the elements $a_j$ are classical symbols of order $m$.

\begin{prop}
\label{signlogoperator}
The operator $\Dsla_{\mathrm{log}}$ is a log-polyhomogeneous pseudo-differential operator of order $(0,1)$ with leading order symbol $c_S(\xi)|\xi|_h^{-1}\log|\xi|_h$ where $c_S$ denotes the Clifford multiplication on $S$. In particular, for any $s>0$ and $a\in C^\infty(M)$, the commutator $[\Dsla_{\mathrm{log}},a]$ is a pseudo-differential operator of order $-1+s$.
\end{prop}

\begin{proof}
By \cite[Proposition 3.4.(2)]{logpolylesch} the log-polyhomogeneous pseudo-differential operators form a bi-filtered algebra. Since $\Dsla|\Dsla|^{-1}\in \Psi^0_{cl}(M,S)=\Psi^{0,0}(M,S)$ has principal symbol $c_S(\xi)|\xi|_h^{-1}$ we have that $\Dsla_{\mathrm{log}}\in \Psi^{0,1}(M,S)$ provided $\log|\Dsla|\in \Psi^{0,1}(M,S)$ with leading order symbol $\log|\xi|_h$. It suffices to prove that $\log\Delta\in \Psi^{0,1}(M,S)$ with leading order symbol $\log|\xi|_h^2$ for any Laplacian type operator $\Delta$ on $S\to M$. This fact follows from \cite[Equation (6), Page 121]{grubblogproj}.

It follows from the composition formulas for log-polyhomogeneous pseudo-differential operators (see \cite[Proposition 3.4.(2)]{logpolylesch}) that the leading order term in the symbol of $[\Dsla_{\mathrm{log}},a]$ is the Poisson bracket $\{c_S(\xi)|\xi|_h^{-1}\log|\xi|_h,a\}_{T^*M}$ which is of order $(-1,1)$. Thus, $[\Dsla_{\mathrm{log}},a]\in \Psi^{-1,1}(M,S)\subseteq \Psi^{-1+s}(M,S)$ for all $s>0$.
\end{proof}

\begin{cor}
The spectral triple $(C^\infty(M), L^2(M,S),\Dsla_{\mathrm{log}})$ satisfies
\begin{enumerate}
\item the class $[C^\infty(M), L^2(M,S),\Dsla_{\mathrm{log}}]$ coincides with $[C^\infty(M), L^2(M,S),\Dsla]$ in $K_*(M)$;
\item $(C^\infty(M), L^2(M,S),\Dsla_{\mathrm{log}})$ is $\theta$-summable;
\item for $a\in C^\infty(M)$,  $[\Dsla_{\mathrm{log}},a]\in \mathcal{L}^{d+s}(L^2(M,S))$ for any $s>0$, where $d=\dim(M)$. 
\end{enumerate}
\end{cor}
Thus, $[\Dsla_{\mathrm{log}},a]$ is not only bounded but even Schatten class and in particular it is a compact operator.

\begin{example}
\label{moresolenoidsandisometries}
We return to the circle, with an eye towards solenoids (cf. Example \ref{solenoidsandisometries}). We let $\Dsla:=\frac{1}{2\pi i}\frac{\mathrm{d}}{\mathrm{d} x}$ on $L^2(S^1)$ (with periodic boundary conditions). Then $\Dsla e_k=ke_k$ so 
$$\Dsla_{\mathrm{log}}e_k=\begin{cases}
0, \quad &k=0,\\
\mathrm{sign}(k)\log|k|e_k, \quad &k\neq 0.
\end{cases}$$
In particular, with $z\in C^\infty(S^1)$ denoting the complex coordinate function $x\mapsto \mathrm{exp}(2\pi i x)$,
$$[\Dsla_{\mathrm{log}},z]e_k=
\begin{cases}
\log(2)e_0, \quad &k=-1,\\
\mathrm{sign}(k)\log\left(1+\frac{2|k|+1}{k^2+1}\right)e_{k+1}, \quad &k\neq -1.
\end{cases}$$
Since $\log(1+t)=-\sum_{k=1}^\infty (-t)^k/k$ for $|t|<1$,  $[\Dsla_{\mathrm{log}},z]\in \Psi^{-1}(S^1)$ (see \cite{agran84}). This gives us a proof of a refinement of the commutator property in Proposition \ref{signlogoperator}.(3) for $S^1$. Any $a\in C^\infty(S^1)$ admits an expansion $a=\sum_{k\in \Z}a_kz^k$ for a rapidly decreasing sequence $(a_k)_{k\in \Z}$ so $[\Dsla_{\mathrm{log}},a]\in \Psi^{-1}(S^1)$. We remark that if $\dim(M)>1$, the property $[\Dsla_{\mathrm{log}},a]\in \Psi^{-1}(M)$ for all $a\in C^\infty(M)$ fails.
\end{example}

\begin{prop}
Let $M$ be a $d$-dimensional manifold, $g:M\to M$ an $n$-fold smooth covering map acting conformally on a metric $h_M$ that lifts to a Clifford bundle $S\to M$, and $V$ the isometry on $L^2(M,S)$ constructed as in Proposition \ref{theisometry}. Assume that $g_S:g^*S\to S$ is Clifford linear. If $\Dsla$ is a Dirac operator on $S$, the operator $[\Dsla_{\mathrm{log}},V]$ admits a bounded extension to $L^2(M,S)$. 
\end{prop}

\begin{proof}
By Proposition \ref{signlogoperator}, $[\Dsla_{\mathrm{log}},V]$ is bounded if and only if $[\Dsla_{\mathrm{log}},V_0]$ is bounded. It in fact suffices to prove that $\chi[\Dsla_{\mathrm{log}},V_0]\chi'$ is bounded for any $\chi,\chi'\in C^\infty(M)$ that are supported in balls $U,U'\subseteq M$ and such that $g|_{U'}$ is injective and $g(U')\subseteq U$. Since $\Dsla_{\mathrm{log}}$ is a log-polyhomogeneous pseudo-differential operator by Proposition \ref{signlogoperator}, the change of coordinates formula for such (see \cite[Proposition 3.5]{logpolylesch}) implies that $\chi[\Dsla_{\mathrm{log}},V_0]\chi'$ is a Fourier integral operator with log-polyhomogeneous symbol of order $(0,-1)$. In fact, the leading order term in the symbol of $\chi[\Dsla_{\mathrm{log}},V_0]\chi'$ is of order $(0,0)$ since 
$$c_S(Dg^t\xi)|Dg^t\xi|_h^{-1}\log|Dg^t\xi|_h-c_S(\xi)|\xi|_h^{-1}\log|\xi|_h=\frac{1}{2}c_S(\xi)|\xi|_h^{-1}\log N.$$
In this identity we use that $c_S(\xi)|\xi|^{-1}$ is invariant under conformal changes of metric and that $|Dg^t\xi|_h=N^{1/2}|\xi|_h$. This computation shows that $\chi[\Dsla_{\mathrm{log}},V_0]\chi'$ is a Fourier integral operator of order $0$, hence is bounded.
\end{proof}

\begin{example}
\label{evenmoresolenoidsandisometries}
We return to the example of solenoids on $S^1$ from the Examples \ref{solenoidsandisometries} and \ref{moresolenoidsandisometries}. In this case, we have 
$$[\Dsla_{\mathrm{log}},V]e_k=\begin{cases}
0, \quad &k=0,\\
\mathrm{sign}(k)\log|n|e_{nk}, \quad &k\neq 0.
\end{cases}$$
So $\|[\Dsla_{\mathrm{log}},V]\|_{\mathbb{B}(L^2(S^1))}= \log|n|$.
\end{example}

\begin{thm}
\label{liftingonsole}
Let $M$ be a $d$-dimensional manifold, $g:M\to M$ an $n$-fold smooth covering map acting conformally on a metric $h_M$ that lifts to a Clifford bundle $S\to M$, $\Dsla$ a Dirac operator on $S$ and $\pi_V:O_{E_g}\to \mathbb{B}(L^2(M,S))$ the representation from Proposition \ref{constrrep}. Assume that $g_S:g^*S\to S$ is Clifford linear. We let $\mathcal{A}\subseteq O_{E_g}$ denote the dense pre-image of the $*$-algebra generated by $C^\infty(M)$ and $V$ under $\pi_V$. The data $(\mathcal{A},L^2(M,S),\Dsla_{\mathrm{log}})$ is a spectral triple for $O_{E_g}$ such that, under the inclusion $\iota:C(M)\to O_{E_g}$,
$$\iota^*[\mathcal{A},L^2(M,S),\Dsla_{\mathrm{log}}]=[C^\infty(M),L^2(M,S),\Dsla]\in K_*(M).$$
\end{thm}

The following corollary is immediate from Theorem \ref{liftingonsole} and Corollary \ref{corollarofpseqonpage14}.

\begin{cor}
The class $[\mathcal{A},L^2(M,S),\Dsla_{\mathrm{log}}]\in K^*(O_{E_g})$ is in the image of the Cuntz-Pimsner mapping $K^{*+1}(C(M))\to K^*(O_{E_g})$ if and only if $[S]=0$ in the $K$-theory group of (graded) Clifford bundles.
\end{cor}

\begin{example}
Let us consider the $K$-homological consequences of the constructions in the solenoid example (see Examples \ref{solenoidsandisometries}, \ref{moresolenoidsandisometries} and \ref{evenmoresolenoidsandisometries}). We assume $n>1$ to ensure that $g$ is oriented and lifts to the spinor bundle. By combining the $K$-theory computations in \cite[Proposition 4.3]{yisolenoid} with the duality results of Jerry Kaminker, Ian Putnam and the fourth listed author \cite{KPW} (see \cite[Page 293]{yisolenoid}) one computes that $K^0(O_{E_g})=\Z$ and $K^1(O_{E_g})=\Z\oplus \Z/(n-1)\Z$. Another approach relating these $K$-homology groups with the $K$-homology of $S^1$ is to use Corollary \ref{corollarpseqpage13} giving rise to the six term exact sequences
\begin{equation}
\label{sixtermforsonesol}
\begin{CD}
K^0(O_{E_g})@>\iota^*>>K^0(C(S^1))@>1-[E_g]>>K^0(C(S^1)) \\
@A[L^2(\G_g)_{C(V)},D]\otimes -AA  @. @VV[L^2(\G_g)_{C(V)},D]\otimes -V\\
K^1(C(S^1)) @<1-[E_g]<<K^1(C(S^1))@<\iota^*<<K^1(O_{E_g})\\
\end{CD}
\end{equation}
Some short computations show that the diagrams 
\[\begin{CD}
K^0(C(S^1))@>1-[E_g]>>K^0(C(S^1)) \\
@V\cong VV  @VV\cong V\\
\Z @>1-n>> \Z\\
\end{CD} \qquad\mbox{and}\qquad \begin{CD}
K^1(C(S^1))@>1-[E_g]>>K^1(C(S^1)) \\
@V\cong VV   @VV\cong V\\
\Z @>0>>\Z\\
\end{CD} \]
commute.
Using these diagrams, the diagram \eqref{sixtermforsonesol} collapses to 
\[\begin{CD}
K^0(O_{E_g})@>0>>\Z@>1-n>>\Z \\
@A[L^2(\G_g)_{C(V)},D]\otimes -AA  @. @VV[L^2(\G_g)_{C(V)},D]\otimes -V\\
\Z @<0<<\Z@<\iota^*<<K^1(O_{E_g})\\
\end{CD}\]
The generator of the upper right corner is a character $\mathrm{ev}_\theta:C(S^1)\to \C$ for a $\theta\in \R/\Z$. It follows that $K^1(O_{E_g})$ is generated by any pre-image $x$ of the fundamental class $[S^1]=[L^2(S^1),\Dsla]\in K^1(C(S^1))$ and the product $[L^2(\G_g)_{C(V)},D]\otimes_{C(S^1)} [\mathrm{ev}_\theta]$. Clearly, such an $x$ is of infinite order and $[L^2(\G_g)_{C(V)},D]\otimes_{C(S^1)} [\mathrm{ev}_\theta]$ is of order $n-1$. By Theorem \ref{liftingonsole}, $x$ can be taken to be represented by the spectral triple $(\mathcal{A},L^2(S^1),\Dsla_{\mathrm{log}})$ for $O_{E_g}$. The $(n-1)$-torsion class given by the product $[L^2(\G_g)_{C(V)},D]\otimes_{C(S^1)} [\mathrm{ev}_\theta]$ is represented by the spectral triple $(\mathcal{A},L^2(\G_g)\otimes_{\mathrm{ev}_\theta}\C,D\otimes_{\mathrm{ev}_\theta}1)$. Note that $L^2(\G_g)\otimes_{\mathrm{ev}_\theta} \C\cong \ell^2(\mathcal{V}_\theta)$ where the discrete set $\mathcal{V}_\theta$ is given by
$$\mathcal{V}_\theta=\{(x,j)\in S^1\times \Z: \;(x,j,\theta)\in \G_g\}\cong \Z\left[\frac{1}{n}\right]/\Z.$$
For $f\in c_c(\mathcal{V}_\theta)$, the operator $D\otimes_{\mathrm{ev}_\theta}1$ can be described by 
$$(D\otimes_{\mathrm{ev}_\theta}1)f(x,j)=\psi(n,\kappa(x,j,\theta))f(x,j).$$
It is an interesting problem to compute the product $[L^2(\G_g)_{C(S^1)},D]\otimes_{C(S^1)}[S^1]$ in order to find the generator of $K^0(O_{E_g})$. 
\end{example}

We now turn to describing a representative of the Kasparov product of the class in Theorem \ref{liftingonsole} with the Morita equivalence in Theorem \ref{moreth}. Our geometric setup presents us with a natural candidate for this product, for which we can check the sufficient conditions in Kucerovsky's theorem \cite[Theorem 13]{Kuc}.

As above, we define $X:=X_M$ and let $P\subseteq X$ be a finite $\varphi$-invariant set of periodic points. Since $\pi_0:X^u(P)\to M$ is a local homeomorphism, we can equip $X^u(P)$, $G^s(P)$ and $Z^u(P)$ with smooth structures compatible with the groupoid operations. Since $g$ lifts to a unitary action on $S$, we can integrate the action of the groupoid $G^s(P)$ on $X^u(P)$ to an action of $C^*(G^s(P))\rtimes \Z$ on $L^2(X^u(P),\pi_0^*(S))$, where $X^u(P)$ is equipped with the measure defined from the metric $\tilde{h}_{X^u(P)}:=\pi_0^*h_M$. The isomorphism 
$$L^2(X^u(P),\pi_0^*(S))\cong L^2(Z^u(P))_{C^*(\G_g)}\otimes_{C^*(\G_g)}L^2(M,S),$$ 
induced from Proposition \ref{tensorfactor} is compatible with the left $C^*(G^s(P))\rtimes \Z$-action. We have that $\varphi^*\tilde{h}_{X^u(P)}=\pi_0^*(N)\tilde{h}_{X^u(P)}$. Let $\tilde{\Dsla}$ denote the lift of $\Dsla$ to $X^u(P)$; it is a Dirac operator on $\pi_0^*(S)\to X^u(P)$. Using the proof of Proposition \ref{signlogoperator} and the fact that $X^u(P)$ is a complete manifold of bounded geometry, we can deduce that $\tilde{\Dsla}_{\rm log}$ is a self-adjoint log-polyhomogeneous pseudo-differential operator on $X^u(P)$.

\begin{prop}
\label{anotherproduct}
Let $M$, $g$, $h_M$, $S$ and $\Dsla$ be as in Theorem \ref{liftingonsole}. Assume that $g:V\to V$ is an irreducible open surjection satisfying Wieler's axioms. The data $(C^\infty_c(G^s(P))\rtimes_{\rm alg}\Z, L^2(X^u(P),\pi_0^*(S)), \tilde{\Dsla}_{\rm log})$ is a spectral triple for $C^*(G^s(P))\rtimes \Z$ representing the product of the class in Theorem \ref{liftingonsole} with the Morita equivalence in Theorem \ref{moreth}.
\end{prop}

\end{document}